\documentclass{article}
\pdfoutput=1 
\usepackage[utf8]{inputenc}
\usepackage{amsthm}
\usepackage{amsmath}
\usepackage[msc-links]{amsrefs}
\usepackage{hyperref}
\usepackage{amsfonts}
\usepackage{mathtools}
\usepackage[a4paper, total={6in, 8in}]{geometry}
\usepackage{setspace}
 \usepackage{helvet}
 
\newtheorem{theorem}{Theorem}[section]

\newtheorem{lemma}[theorem]{Lemma}
\newtheorem{proposition}[theorem]{Proposition}
\theoremstyle{definition}
\newtheorem{definition}{Definition}[section]
\theoremstyle{remark}
\newtheorem*{remark}{Remark}

\newcommand{\defeq}{\vcentcolon=}

\newcommand{\tb}[1]{\frac{\partial}{\partial \theta _{#1}}}

\newcommand{\sphere}{ \partial B}
\newcommand{\intsphere}{ \int_{\partial B}}

\newcommand{\minus}{\scalebox{0.75}[1.0]{$-$}}
\usepackage{tensor}

\title{Quantitative Quermassintegral Inequalities for Nearly Spherical Sets}
\author{Caroline VanBlargan, Yi Wang }
\date{}
\theoremstyle{break}

\begin{document}

\maketitle

\begin{abstract}
   In this paper, we establish quantitative Alexandrov-Fenchel inequalities for quermassintegrals on nearly spherical sets. In particular, we bound the $(k,m)$-isoperimetric deficit from below by the Frankael asymmetry. We also find a lower bound on the $(k,m)$-isoperimetric deficit using the spherical deviation.
\end{abstract}
\section{Introduction}
In this paper, we study stability in the quermassintegral inequalities for nearly spherical sets, which is largely motivated by work done on the stability in the classical isoperimetric inequality. In particular, we are motivated by analysis done on the \textit{isoperimetric deficit} $\delta(\Omega)$ of a domain $\Omega \subseteq \mathbb{R}^{n+1}$, defined as
\begin{align}
\label{classical_iso_def}
    \delta(\Omega) := \frac{P(\Omega) - P(B_{\Omega})}{P(B_{\Omega})}.
\end{align}
Here $|B_{\Omega}|$ is the volume of $B_{\Omega}$, $B_{\Omega}$ is a ball such that $|\Omega| = |B_{\Omega}|$, and $P( \cdot )$ gives the perimeter of a set. The classical isoperimetric inequality is equivalent to $\delta(\Omega) \geq 0$, with equality if and only $\Omega$ is a ball. There has been a lot of work studying quantitative isoperimetric inequalities inspired by the \textit{Bonnesen type inequalities}, which was named by Osserman in \cite{MR519520}. This was based off work by Bonnesen, where he studied inequalities in the form
\begin{align}
    L^2- 4\pi A \geq \lambda(C).
\end{align}
In this setting, $L$ and $A$ represent the length and area enclosed by a simple closed curve $C$ in $\mathbb{R}^2$. Moreover, $\lambda(C)$ satisfies three conditions:
\begin{enumerate}
    \item  $\lambda(C) \geq 0$.
    \item  $\lambda(C) = 0$ precisely when $C$ is a circle.
    \item   $\lambda(C)$ measures geometrically how close $C$ is to a circle. 
\end{enumerate}
Fuglede worked to expand these results to higher dimensions in \cite{MR859955} and \cite{MR942426}, where they proved a quantitative isoperimetric inequality for nearly spherical sets. They used this result to study the stability for convex domains using the spherical deviation. 
\begin{definition}
\label{sphericaldeviation}
For a domain $\Omega \subseteq \mathbb{R}^{n+1}$, set
    $\tilde{\Omega}:= \frac{\text{Vol}(B)}{\text{Vol}(\Omega)}(\Omega - \text{bar}(\Omega))$,
where $\text{bar}(\Omega)$ is the barycenter of $\Omega$ and $B$ is the unit ball. 
The \textit{spherical deviation} of $\Omega$, $d(\Omega)$, is defined as
\begin{align}
   d(\Omega) := 
  d_H ( \tilde{\Omega}, B ),
\end{align}
where $d_H(\cdot, \cdot)$ gives the Hausdorff distance between two sets.
\end{definition}
\noindent
Specifically, for a convex domain $\Omega$, they established an inequality in the form $d(\Omega) \leq f(\delta(\Omega))$, for some function $f$.

To establish a quantitative isoperimetric inequality for more general domains, the \textit{Fraenkel asymmetry}, $\alpha(\Omega)$, is a well-studied quantity used as a lower bound. 
\begin{definition}
\label{frankael_def}
Suppose $\Omega \subseteq\mathbb{R}^{n+1}$. The \textit{Fraenkel asymmetry} of $\Omega$ is denoted by $\alpha(\Omega)$, where
\begin{align}
    \alpha(\Omega):= \inf \bigg \{\frac{|\Omega \Delta (x + B_{\Omega})|}{|B_{\Omega}|}, x \in \mathbb{R}^{n+1} \bigg\}.
\end{align}
$B_{\Omega}$ denotes the ball centered at the origin with the same volume as $\Omega$, and $\Delta$ denotes the symmetric difference between two sets.
\end{definition}
Using the Fraenkal asymmetry in the study of stability brings us to the \textit{quantitative isoperimetric inequality}, which asks if there is a fixed $C(n)>0$  such that all Borel sets $\Omega\subseteq \mathbb{R}^{n+1}$ with finite measure satisfy the inequality
\begin{align}
    \delta(\Omega) \geq C(n) \alpha^{m}(\Omega), \label{stability}
\end{align}
for some exponent $m$. The quantitative isoperimetric inequality $ \delta(\Omega) \geq C \alpha^{2}(\Omega)$ was shown for Steiner symmetrical sets in \cite{MR1243100} by  Hall, Hayman, and Weitsman. Later, by using results on the Steiner symmetral, Hall showed
\eqref{stability} in \cite{MR1166511}, but with a suboptimal exponent of $m=4$.

 Finally, in \cite{MR2456887}, Fusco, Maggi, and Pratelli showed, by using symmetrizations of $\Omega$, that \eqref{stability} holds with optimal expontent $m=2$ for Borel sets $\Omega\subseteq \mathbb{R}^{n+1}$ of finite measure. Figalli, Maggi, and Pratelli in  \cite{MR2672283} proved this optimal result in the more general setting of the anisotropic perimeter.
 \begin{theorem}
 [\cite{MR2672283}, \cite{MR2456887}]
 Suppose $n\geq 1$. Then for any Borel set $\Omega\subseteq \mathbb{R}^{n+1}$ of finite measure,
 \begin{align}
     \delta(\Omega) \geq C(n) \alpha^{2}(\Omega), \label{fuscoresult}
 \end{align}
 where $C(n) >0$ depends only on $n$.
 \end{theorem}
\noindent This theorem was proved in \cite{MR2456887} by reducing the problem to $n$-symmetric sets and then using the method of Steiner symmetry. To prove the result in \cite{MR2672283}, the authors did not use symmetrization arguments as done previously. Instead, they applied the Brenier map to employ methods in mass transportation theory. For further reading on the quantitative isoperimetric inequality see \cite{MR3404715} and \cite{MR2402947}.

In this paper, we will be working with quantities motivated by the theory of mixed volumes in convex geometry. For a convex body $\Omega \subset \mathbb{R}^{n+1}$, the \textit{k}-th \textit{quermassintegral} of $\Omega$ is the mixed volume
\begin{align}
    W_k(\Omega) := V(\Omega,...,\Omega,B,...,B),
\end{align}
where $\Omega$ appears in the first $n+1-k$ entries and $B$, which is the unit ball in $\mathbb{R}^{n+1}$, appears in the last $k$ entries. The famous Steiner formula states that the volume of $\Omega + tB$ is a polynomial in $t$. In particular,
\begin{align}
    \text{Vol}(\Omega + tB) = \sum_{k = 0}^{n+1} {n +1 \choose k}W_k(\Omega)t^k.
    \end{align}
Next, denoting $\omega_m$ as the volume of the unit $m$-ball, we set
\begin{align}
    V_k(\Omega) &:= \frac{\omega_k}{\omega_{n+1}}W_{n+1-k}(\Omega).
\end{align}
 Note that $V_{n+1}(A) = \text{Vol}(\Omega)$ and $V_n(\Omega)= \frac{\omega_{n+1}}{(n+1)\omega_n}\text{Area}(\partial \Omega)$. We obtain, as a consequence of the Alexandrov-Fenchel inequalities, the \textit{quermassintegral inequalities}
\begin{align}
   \bigg( \frac{V_{k+1}(\Omega)}{V_{k+1}(B)}\bigg)^{\frac{1}{k+1}} \leq \bigg( \frac{V_{k}(\Omega)}{V_{k}(B)}\bigg)^{\frac{1}{k}}.\label{alek}
\end{align}
When $k=n$, the inequality in \eqref{alek} is simply the classical isoperimetric inequality. For convex domains and $k\geq 1$, quermassintegrals have the useful integral formula
\begin{align}
\label{mixedvolumeidentity}
    V_{n+1 -k} = \frac{(n+1-k)!(k-1)!}{(n+1)!}\frac{\omega_{n+1-k}}{\omega_{n+1}}\int_{M}  \sigma_{k-1}(L) d\mu,
\end{align}
where  $M:= \partial \Omega$, $L$ is the second fundamental form of $M$, and $\sigma_k(L)$ is the \textit{$k$-th mean curvature} of $M$. The $k$-th mean curvature is the $k$-th elementary symmetric polynomial of the principal curvatures. 
The inequalities in \eqref{alek} equivalently state that for convex domains
\begin{align}
    \bigg (\int_M \sigma_{k-1}(L) d\mu\bigg )^{\frac{1}{n-k+1}}
    \leq
    C(n,k) \bigg (\int_M\sigma_{k}(L) d\mu \bigg )^{\frac{1}{n-k}},
    \label{aleksigma}
\end{align}
where $C(n,k)$ is the constant that gives equality in the case where $M$ is a sphere.

 Much of the previous work to establish \eqref{aleksigma} relies heavily on working with convex domains. There has been work extending \eqref{aleksigma} to a class of nonconvex domains known as \textit{k-convex domains}, where $\sigma_j(L) \geq 0$ for $1 \leq j \leq k$. Guan and Li showed in \cite{MR2522433}
that \eqref{aleksigma} holds in $k$-convex, starshaped domains. To prove the inequalities, they used a normalization of the flow $ X_t = \frac{\sigma_k(L)}{\sigma_{k-1}(L)} \nu$, which was studied by Gerhardt in \cite{MR1064876} and Urbas in \cite{MR1082861}.

In \cite{MR3107515}, Chang and Wang  were able to show \eqref{aleksigma} without the requirement of a starshaped domain, but with the added assumption of having $(k+1)$-convexity instead of just $k$-convexity, and the constant $C(n,k)$ is non-optimal. They proved this using optimal transport methods. See also \cite{MR3291634}, \cite{MR2831436}, and \cite{MR3247383}.

Our study of stability in the quermassintegral inequalities is inspired by work done by Fuglede in \cite{MR942426} and by Cicalese and Leonardi in \cite{MR2980529}.
 \begin{definition}
 \label{nearlysphere}
Suppose $M = \{ (1+u(x))x: x \in \partial B  \}$, where $\partial B$ is the unit sphere in $\mathbb{R}^{n+1}$ and $u:\partial B \rightarrow (-1, \infty)$ is a smooth function on the unit sphere. $M$ is referred to as a \textit{nearly spherical set} when we have suitable, small bounds on $|u|,$ $|\nabla u|$, and $|D^2u|$.
\end{definition}
\begin{remark}
Nearly spherical sets in \cite{MR859955} and \cite{MR2980529} only require bounds on $|u|$ and $|\nabla u|$. However, since we will be working with curvature terms, we will require small bounds on $|D^2u|$ as well. 
\end{remark}
In \cite{MR2980529}, Cicalese and Leonardi introduced a new method to show \eqref{fuscoresult} for all Borel sets of finite measure. In this paper, they utilized the results in  \cite{MR859955}, which they reformulated by assuming $||u||_{W^{1,\infty}} < \epsilon$ and found
\begin{align}
     \delta(\Omega) \geq  \frac{1 - O(\epsilon)}{2}||u||_{L^2}^2 + \frac{1}{4}|| \nabla u||^2_{L^2}, \label{cicalese_stable}
\end{align}
where $u$ is the function from Definition \ref{nearlysphere}. Note that functions in $O(\epsilon)$ may obtain either positive or negative values. 
It quickly follows that for nearly spherical sets, 
 \begin{align}
 \label{sphericaliso}
     \delta(\Omega) \geq C(1 + O(\epsilon))\alpha^2(\Omega).
 \end{align}
They proved the \textit{Selection Principle}, which provided a new proof of the quantitative isoperimetric inequality by reducing the problem to nearly spherical sets converging to the unit ball. 

To get \eqref{sphericaliso} from \eqref{cicalese_stable}, only the weaker statement
$ \delta(\Omega) \geq  \frac{1 - O(\epsilon)}{2}||u||_{L^2}^2$ is needed. However, in \cite{MR3216839}, Fusco and Julin showed a stronger result for stability in the isoperimetric problem, where they bound the \textit{asymmetry index} $A(\Omega)$, so that $\delta(\Omega) \geq C A^2(\Omega) $. To do this, they need the full result that $\delta(\Omega) \geq C||u||_{W^{1,2}}^2 $ for nearly spherical sets. Then, as in \cite{MR2980529}, they are able to use the methods of the Selection Principle to reduce the general problem to the results for nearly spherical sets. 

In this paper we aim to prove a version of \eqref{cicalese_stable} as it applies to the quermassintegral inequalities. First, we define $I_k(\Omega)$ by integrating the $k$-th mean curvature of $M$ for $k \geq 0$,  that is
\begin{align}
     I_k(\Omega) := \int_{M} \sigma_k(L)  \textit{ dA} . 
\end{align}
Also, for $k=-1$ we define
\begin{align}
    I_{\minus 1}(\Omega) := \text{Vol}(\Omega). 
\end{align}
We are now able to define a natural generalization of the isoperimetric deficit for quermassintegrals.
\begin{definition}
\label{km_deficit}
For $\minus 1<k \leq n$ and $\minus 1 \leq m < k$, the \textit{(k,m)-isoperimetric deficit} is denoted by $\delta_{k,m}(\Omega)$, where
\begin{align}
    \delta_{k,m}(\Omega) \defeq \frac{I_k(\Omega) - I_k(B_{\Omega,m})}{I_k(B_{\Omega,m})}. 
    \end{align}
Here $B_{\Omega,m}$ is the ball centered at the origin where $I_m(B_{\Omega,m})= I_m(\Omega)$. 
\end{definition}

In Section 3, we compute an explicit formula for the $k$-th mean curvature of nearly spherical sets. Then in Section 4, we add in the assumption that $||u||_{W^{2,\infty}} < \epsilon$ to expand out $I_k(\Omega)$. This brings us to our first main theorem. 
\begin{theorem} 
\label{firstthm}
\textit{ Suppose  $\Omega= \{ (1+u(\frac{x}{|x|}))x: x \in B  \} \subseteq \mathbb{R}^{n+1}$, where $u \in C^3(\partial B)$, $\text{Vol}(\Omega) = \text{Vol}(B)$, and $\text{bar}(\Omega) =0$.}
 For all $\eta >0$, there exists $\epsilon >0$ such that if $||u||_{W^{2,\infty}}< \epsilon$, then
\begin{align}
    \delta_{k,j}(\Omega)
    \geq 
\bigg (
\frac{(n-k)(k+1)}{2n(n+1)^2}  -\eta \bigg)\alpha^2(\Omega).
\end{align}
\end{theorem}
In Section 5, we prove a similar theorem, but we assume that $I_j(\Omega) = I_j(B)$ instead of $\text{Vol}(\Omega) = \text{Vol}(B)$.
\begin{theorem} 
\label{secondthm}
Fix $0 \leq j < k$. Suppose $\Omega= \{ (1+u(\frac{x}{|x|}))x: x \in B  \}\subseteq \mathbb{R}^{n+1}$, where $u \in C^3(\partial B)$,  $I_j(\Omega) = I_j(B)$, and $\text{bar}(\Omega) =0$.
For all $\eta>0$, there exists $\epsilon>0$ such that if  $||u||_{W^{2,\infty}}< \epsilon$, then
\begin{align}
    \delta_{k,j}(\Omega)
    \geq 
\bigg (
\frac{n(n-k)(k-j)}{4(n+1)^2}  - \eta \bigg)\alpha^2(\Omega).
\end{align}

\end{theorem}
 We remark that for sufficiently small $||u||_{W^{2,\infty}}$, $\Omega$ is a convex domain. Then, we already know from the result of Guan and Li, which assumes $\Omega$ is k-convex and starshaped, that $\delta_{k,j}(\Omega) \geq 0$. So, we are establishing a quantitative isoperimetric inequality in this case.

In Section \ref{haussdorffsection}, we prove the following theorem.

\begin{theorem}
\label{infboundintro}
Suppose  $\Omega = \{ (1+u(\frac{x}{|x|}))x: x \in B  \}\subseteq \Omega$, where $u \in C^3(\partial B)$, $\text{Vol}(\Omega) = \text{Vol}(B)$, and   \text{bar}($\Omega$) = 0.
 There exists an $\eta >0$ so if  $||u||_{W^{2,\infty}} < \eta$, then
  \[ ||u||_{L^{\infty}}^n 
  \leq
  \begin{cases} 
          C \delta_{k,\minus 1}^{1/2}(\Omega) & n = 1 \\
        C \delta_{k,\minus 1}(\Omega)\log \frac{A}{\delta_{k,\minus 1}(\Omega)} & n=2\\
          C\delta_{k,\minus 1}(\Omega) & n \geq 3 ,
       \end{cases}
    \]
  where $A,C>0$ depend only on n.   
\end{theorem}

Theorem \ref{infboundintro} shows that the $(k,\minus 1)$-deficit gives a control on $||u||_{L^{\infty}}$, which is equivalent to the spherical deviation $d(\Omega)$. The proof of this theorem follows closely to Fuglede's in \cite{ MR942426}, where they proved this theorem when $k=0$ (for the classical isoperimetric deficit). In \cite{MR2994695}, Fusco, Gelli, and Pisante expanded this stability result for domains where they impose a uniform cone condition on the boundary. Studying the control on $||u||_{L^{\infty}}$ gives a stronger result than just the Fraenkel asymmetry, although there must be some regularity condition imposed for it to hold. In \cite{MR942426}, bounding $||u||_{L^{\infty}}$ by the classical isoperimetric deficit for nearly spherical domains was a key result to establish stability results for convex domains (without the the assumption that the domain is nearly spherical). We hope Theorem \ref{infboundintro} will be  useful for establishing a stability result with the spherical deviation for less restrictive $k$-convex domains.

\section{Preliminaries}
\subsection{The  $k$-th mean curvature} \label{sigmaproperty}
For $\lambda = (\lambda_1, ..., \lambda_n) \in \mathbb{R}^n$, we denote $\sigma_k(\lambda)$ as the $k$\textit{-th elementary symmetric polynomial} of $(\lambda_1, ..., \lambda_n)$. That is, for $1 \leq k \leq n$, 

\begin{align}
\sigma_k(\lambda) = \sum_{i_1<i_2<...<i_n} \lambda_{i_1} \lambda_{i_2} \cdot \cdot \cdot \lambda_{i_k},
\end{align}
and 
\begin{align}
\sigma_0(\lambda) = 1.
\end{align}

This leads to a natural generalization of the mean curvature of a surface. 

\begin{definition}{}
Suppose $\Omega$ is a smooth, bounded domain in $\mathbb{R}^{n+1}$. For $x \in M:=\partial \Omega$, the $k$\textit{-th mean curvature} of $M$ at $x$ is 
$\sigma_{k}(\lambda)$, where $\lambda = (\lambda_1(x), ..., \lambda_n(x))$ are the principal curvatures of $M$ at $x$. 
\end{definition}
Observe that in this definition, $\sigma_1(\lambda)$ is the mean curvature and $\sigma_n(\lambda)$ is the Gaussian curvature. 
When $(\lambda_1,..., \lambda_n)$ are the eigenvalues of a matrix $A = \{A_{j}^i \}$, we denote $\sigma_k(A) = \sigma_k(\lambda)$, which can be equivalently calculated as
\begin{align}
\sigma_k(A) = \frac{1}{k!} \delta_{i_1\cdot\cdot\cdot i_k}^{j_1\cdot\cdot\cdot j_k} A_{j_1}^{i_1}\cdot\cdot\cdot A_{j_k}^{i_k},
\end{align}
 using the Einstein convention to sum over repeated indices.
 
 So, if $L$ is the second fundamental form of $ M$, we can use this expression for $\sigma_k(L)$ to compute the $k$-th mean curvature of $ M$. 
Throughout this paper, we will be working with family of surfaces where, for $0<j\leq k$, $\sigma_j(L)\geq 0$ at each point. Such surfaces are called \textit{k-convex}.

\begin{definition}
Let $\Omega $ be a domain in $ \mathbb{R}^{n+1}$. Then the hypersurface $M:= \partial \Omega$ is said to be \textit{strictly k-convex} if the principal curvatures $\lambda =(\lambda_1,..., \lambda_n)$ lie in the \textit{G\r{a}rding cone} $\Gamma^{+}_k$, which is defined as
\begin{align}
    \Gamma^{+}_k := \{ \lambda \subseteq 
    \mathbb{R}^n: \sigma_j(\lambda) > 0, 1 \leq j \leq k\}.
\end{align}
\end{definition}
Note that n-convexity is the same as normal convexity. A useful operator related to $\sigma_k$ is the  \textit{Newton transformation tensor} $[T_k]^j_i$.
\begin{definition}
The Newton transformation tensor, $[T_k]^j_i$, of $n\times n$ matrices $\{ A_{1}, ..., A_k\}$ is defined as
\begin{align}
    [T_k]^j_i(A_1,...,A_k):= \frac{1}{k!} \delta_{ii_1\cdot\cdot\cdot i_k}^{jj_1\cdot\cdot\cdot j_k} (A_1)_{j_1}^{i_1}\cdot\cdot\cdot (A_k)_{j_k}^{i_k}.
\end{align}
When $A_1 = A_2 =...=A_k = A$, we denote $[T_k]^j_i(A) = [T_k]^j_i(A,... ,A)$.
\end{definition}
A related operator is  $\Sigma_k$, which  the polarization of $\sigma_k$.
\begin{definition}
Suppose $\{ A_{1}, ..., A_k\}$ is a collection of $n\times n$ matrices. We denote
    \begin{align}
    \Sigma_k(A_1,...,A_k)
    &:= 
    (A_1)_{j}^i[T_{k-1}]^j_i(A_2,...,A_k)
   \nonumber
   \\ 
   &= \frac{1}{(k-1)!} \delta_{i_1\cdot\cdot\cdot i_k}^{j_1\cdot\cdot\cdot j_k} (A_1)^{i_1}_{j_1}\cdot\cdot\cdot (A_k)^{i_k}_{j_k}.
\end{align}
\end{definition}
Two useful identities are
\begin{align}
    \sigma_k(A) = \frac{1}{k} \Sigma_k(A,...,A) = \frac{1}{k} A_{j}^i[T_{k-1}]^j_{i}(A) ,
\end{align}
and
\begin{align}
    \frac{\partial \sigma_k(A)}{\partial A_{j}^i} = \frac{1}{k}[T_{k-1}]^j_i(A).
\end{align}
We will also use the identity 
\begin{align}
    A_{s}^j[T_m]^i_j(A) = \delta^{i}_s\sigma_{m+1}(A) - [T_{m+1}]_s^i(A).
\end{align}

\subsection{The $(k,m)$-isoperimetric deficit}
\label{deficitprelim}

In this paper we look at a more general notion of the isoperimetric deficit as it pertains to the $k$-th mean curvature. We consider a bounded domain $\Omega \subset \mathbb{R}^{n+1}$ where $M:=\partial \Omega$ is a smooth hypersurface. First, we define $I_k(\Omega)$ by integrating the $k$-th mean curvature of $M$. That is,
\begin{align}
     I_k(\Omega) \defeq \int_{M} \sigma_k(L)  d\mu . 
\end{align}
The definition extends to $k=-1$ so that
\begin{align}
    I_{\minus 1}(\Omega) \defeq \text{Vol}(\Omega).
\end{align}
Furthermore, because $\sigma_0(L) = 1$, 
\begin{align}
    I_{0}(\Omega) =\text{Area}(M).
\end{align}As seen from the identity \eqref{mixedvolumeidentity}, $I_k(\Omega)$  is equal to a quermassintegral of $\Omega$  up to a constant.
We aim to study the $(k,m)$-deficit, $\delta_{k,m}(\Omega)$, from Definition \ref{km_deficit}. Note that $\delta_{0,-1}(\Omega)$ is the classical isoperimetric deficit from \eqref{classical_iso_def}.

For $k\geq 0$,
$
    I_k(B) = {n \choose k} Area (\partial B).
$
If we scale $\Omega$ by a fixed $r>0$, so  $rE = \{rx : x \in\Omega\}$, then 
$
    I_k(rE) = r^{n-k} I_k(\Omega).
$
It follows that
\begin{align}
    I_m(r\Omega) = I_m(rB_{\Omega,m}).
\end{align}
Therefore,  $ \delta_{k,m}(\Omega)$ is invariant under scaling. It is also invariant under translation.

Furthermore, if $r$ is the radius of $B_{\Omega,m}$, then $I_m(\Omega) = { n \choose m} r^{n-m}  Area (\partial B)$, giving
\begin{align}
    r = \bigg (\frac{I_m(\Omega)}{{n \choose m}Area (\partial B)} \bigg)^\frac{1}{n - m}.
\end{align}
Then,
\begin{align}
    I_k(B_{\Omega,m}) &= \bigg (\frac{I_m(\Omega)}{{n \choose m}Area (\partial B)} \bigg)^\frac{n-k}{n - m} {n \choose k}Area (\partial B)
   \nonumber \\
    &= \frac{ {n \choose k}}{ {n \choose m}^{\frac{n-k}{n - m}}}Area (\partial B)^{\frac{k-m}{n-m}}
    I_m(\Omega)^{\frac{n-k}{n-m}}.
\end{align}
In particular, when setting $m = k-1$,
\begin{align}
    \delta_{k,k-1}(\Omega) = \frac{I_k(\Omega)}{\frac{n-k+1}{k}I_{k-1}(B)^\frac{1}{n-k+1}I_{k-1}(\Omega)^\frac{n-k}{n-k+1}} - 1.
\end{align}
Thus, the inequality $\delta_{k,k-1}(\Omega) \geq 0$ is equivalent to the quermassintegral inequalities in \eqref{aleksigma}.

Our goal is to look at the quantitative isoperimetric inequality in the setting of the $k$-th mean curvature, which we refer to as the \textit{quantitative $(k,m)$-isoperimetric inequality}. That is, we aim to compare $\delta_{k,m}(\Omega)$ to the \textit{Fraenkel asymmetry} of $\Omega$, $\alpha(\Omega)$, which measures how close $\Omega$ is to a ball (see Definition \ref{frankael_def}).

 The Fraenkel asymmetry of a set is invariant under scalings and translations. Therefore, when studying the quantitative $(k,m)$-isoperimetric inequality, i.e. when there is a fixed $C>0$ so that
\begin{align}
    \delta_{k,m}(\Omega) \geq C \alpha^2(\Omega),
\end{align}
we only need to consider sets where $I_m(\Omega) = I_m(B)$ and $\text{bar}(\Omega)=0$.

\subsection{Nearly spherical sets}
\label{sphericalprelim}

   The focus of this paper is to establish the $(k,m)$-isoperimetric inequality for \textit{nearly spherical sets}. Our approach is inspired by Cicalese and Leonardi's work in the classical quantitative isoperimetric inequality for nearly spherical sets in \cite{MR2980529}.
That is, we consider a smooth, bounded domain $\Omega$ that is starshaped with respect to the origin, which is enclosed by $M := \partial \Omega$. We write $M = \{ (1+u(x))x: x \in \partial B  \}$, where $u:\partial B \rightarrow \mathbb{R}$ is a smooth function. The set $M$ is referred to as a \textit{nearly spherical set} when there is a suitable, small bound on $||u||_{W^{2,\infty}}$. In this section, we establish some useful formulas for nearly spherical sets.

We write $\mathbb{R}^{n+1}$ in spherical coordinates with the tangent basis $\{\frac{\partial}{\partial \theta_1},\frac{\partial}{\partial \theta_2}, ..., \frac{\partial}{\partial \theta_n}, \frac{\partial}{\partial r} \}$. Denoting $s_{ij}$ as the metric on the sphere, we have   $<\frac{\partial}{\partial \theta_i}, \frac{\partial}{\partial r} > = 0$, $<\frac{\partial}{\partial r}, \frac{\partial}{\partial r} > = 1$, and $<\frac{\partial}{\partial \theta_i}, \frac{\partial}{\partial \theta_j} > = r^2s_{ij}$.
Set $u_i = \frac{\partial u}{\partial \theta_i}$. Then, $\{e_i\}$ forms a tangent basis of $M$ where
\begin{align}
e_i = \frac{\partial}{\partial \theta_i} + u_i\frac{\partial}{\partial r}. 
\end{align}
We find,
\begin{align}
\label{outernormal_formula}
N = \frac{-\sum_{i=1}^n s^{ij}u_i \frac{\partial}{\partial \theta_j}+ (1+u)^2 \frac{\partial}{\partial r} }{(1+u)\sqrt{|\nabla u|^2 + (1+u)^2}} ,
\end{align}
where  $N$ is the the outward  unit normal on $M$, and the norm $|\nabla u|$ is taken with respect to the standard metric on $\partial B$.
We compute the metric $g_{ij}$ on $M$ as
\begin{align}
g_{ij} = <e_i,e_j> = (1+u)^2s_{ij} + u_iu_j , 
\end{align}
where $<\cdot, \cdot>$ is the standard Euclidean inner product on $\mathbb{R}^{n+1}$. Setting $g^{ij}$ to be the inverse of $g_{ij}$, we have

\begin{align}
g^{ij}  = \frac{s^{ij}}{(1+u)^2} - \frac{1}{(1+u)^2}\frac{u_ku_ls^{ki}s^{lj}}{|\nabla u|^2 + (1+u)^2}  .
\end{align}
We denote $h_{ij}$ as the second fundamental form on $M$. That is, 
$h_{ij}=-<N,\nabla_{e_i}e_j>$, and we form the shape operator $h^i_j$ by
\begin{align}
h^i_j = g^{ik}h_{kj}.
\end{align}
We now explicitly calculate $h^i_j$. First, note
\begin{align}
   \bigg (\nabla_{\frac{\partial}{\partial \theta_i}}\frac{\partial}{\partial \theta_j}
   \bigg)^k
    &=
    \frac{1}{2}\frac{1}{r^2}s^{kl}
    \bigg(
    \partial_i (r^2 s_{il}) + \partial_j(r^2s_{il}) - \partial_l(r^2s_{ij})
    \bigg)
    \\
    &=
    \frac{1}{2}s^{kl}
    \bigg(
    \partial_i s_{il} + \partial_js_{il} - \partial_ls_{ij}
    \bigg)
    \\
    &= \Gamma^k_{ij},
\end{align}
where $\Gamma^k_{ij}$ refers to the Christoffel symbol on $\mathbb{S}^n$, and \begin{align}
   \bigg (\nabla_{\frac{\partial}{\partial \theta_i}}\frac{\partial}{\partial \theta_j}
   \bigg)^r
    &=-rs_{ij}.
\end{align}
We thus obtain
\begin{itemize}
    \item $\nabla_{\frac{\partial}{\partial \theta_i}}\frac{\partial}{\partial \theta_j} = \Gamma^k_{ij}\frac{\partial}{\partial \theta_k}
    -rs_{ij}\frac{\partial}{\partial r}$.\end{itemize}
Similarly,    
    \begin{itemize}
    \item$ \nabla_{\frac{\partial}{\partial \theta_i}}\frac{\partial}{\partial r} 
    =
    \nabla_{\frac{\partial}{\partial r}}\frac{\partial}{\partial \theta_i}
    =
    \frac{1}{r}\frac{\partial}{\partial \theta_i}$,
       \item$ \nabla_{\frac{\partial}{\partial r}}\frac{\partial}{\partial r} = 0$.
\end{itemize}
Then,
\begin{align}
    \nabla_{e_i} e_j 
      &= \nabla_{\frac{\partial}{\partial \theta_i} + u_i\frac{\partial}{\partial r} }
      \bigg(
      \frac{\partial}{\partial \theta_j} + u_j\frac{\partial}{\partial r}
      \bigg)
      \nonumber 
      \\
      &=  \nabla_{\frac{\partial}{\partial \theta_i} }\frac{\partial}{\partial \theta_j} 
      +  \nabla_{\frac{\partial}{\partial \theta_i} }
      \bigg( u_j\frac{\partial}{\partial r} \bigg)
      + u_i\nabla_{ \frac{\partial}{\partial r} }\frac{\partial}{\partial \theta_j}
       + u_i\nabla_{ \frac{\partial}{\partial r} } 
       \bigg(
       u_j\frac{\partial}{\partial r}
       \bigg)
       \nonumber 
       \\
       &=  \nabla_{\frac{\partial}{\partial \theta_i} }\frac{\partial}{\partial \theta_j} 
      + u_j \nabla_{\frac{\partial}{\partial \theta_i} } \frac{\partial}{\partial r}
      +
      \bigg(\frac{\partial^2}{\partial \theta_i\partial \theta_j}
      u
      \bigg)\frac{\partial}{\partial r} 
      + u_i\nabla_{ \frac{\partial}{\partial r} }\frac{\partial}{\partial \theta_j}
       + u_iu_j\nabla_{ \frac{\partial}{\partial r} }\frac{\partial}{\partial r}
       + u_i
       \bigg(\frac{\partial}{\partial r}u_{i}
       \bigg) \frac{\partial}{\partial r}.
        \\
       &=  \Gamma^k_{ij}\frac{\partial}{\partial \theta_k}
    -(1+u)s_{ij}\frac{\partial}{\partial r}
      +\frac{1}{(1+u)}(u_j\tb{i} + u_i\tb{j})
      +
      \bigg(\frac{\partial^2}{\partial \theta_i\partial \theta_j}
      u
      \bigg)\frac{\partial}{\partial r}
\end{align}
So, our expression for $ h_{ij} $ becomes:
\begin{align}
 & -
 \bigg<
 \frac{-s^{pq}u_{p} \frac{\partial}{\partial \theta_q}+ (1+u)^2 \frac{\partial}{\partial r} }{(1+u)\sqrt{|\nabla u|^2 + (1+u)^2}} ,
 \Gamma^k_{ij}\frac{\partial}{\partial \theta_k}
    -(1+u)s_{ij}\frac{\partial}{\partial r}
      +\frac{1}{(1+u)}
      \bigg(u_j\tb{i} + u_i\tb{j}
      \bigg)
      +
      \bigg(\frac{\partial^2}{\partial \theta_i\partial \theta_j}
      u
      \bigg)\frac{\partial}{\partial r}
    \bigg >.
  \end{align}
Thus,
  \begin{align}
h_{ij} &=\frac{1}{\sqrt{|\nabla u|^2 +(1+u)^2}}
 \bigg(
 (1+u)u_k\Gamma^k_{ij}
 + (1+u)^2s_{ij} 
 +
 2u_iu_j
      -  (1+u) 
      \bigg(
 \frac{\partial^2}{\partial \theta_i\partial \theta_j} u
      \bigg)
     \bigg ) 
       \\
         \nonumber
 &=\frac{1}{\sqrt{|\nabla u|^2 +(1+u)^2}} (2u_iu_j
     + (1+u)^2s_{ij} 
      -  (1+u) u_{ij}
    \bigg  ),
\end{align}
where $u_{ij}$ denotes the Hessian of $u$ on $\mathbb{S}^n$.
Set
\begin{align}
D := \sqrt{|\nabla u|^2 +(1+u)^2}.
\end{align}
Then,
\begin{align}
    h^i_{j} 
    &= g^{ik}h_{kj}
     \nonumber
    \\
   &= 
   \bigg(\frac{s^{ik}}{(1+u)^2} - \frac{1}{(1+u)^2}\frac{u_mu_ls^{mi}s^{lk}}{D^2}\bigg)
   \frac{1}{D}
   \bigg(
   2u_ku_j
     + (1+u)^2s_{kj} 
      -  (1+u) u_{kj}
   \bigg)
    \nonumber
    \\
    &=
    \frac{2u^iu_j}{(1+u)^2D} + \frac{\delta^i_j}{D} - \frac{u^i_{j}}{(1+u)D}
    -\frac{2u^iu_j|\nabla u|^2}{(1+u)^2D^3} -\frac{u^iu_j}{D^3}
    +\frac{u^iu_lu^l_{j}}{(1+u)D^3}.
\end{align}
We observe that
\begin{align}
     \frac{2u^iu_j}{(1+u)^2D}   -\frac{2u^iu_j|\nabla u|^2}{(1+u)^2D^3}
     - \frac{u^iu_j}{D^3}
&= \frac{u^iu_j}{D^3},
\end{align}
which yields
\begin{align}
    h^i_{j} 
    &=
   \frac{\delta^i_j}{D} - \frac{u^i_{j}}{(1+u)D}
    +\frac{u^iu_lu^l_{j}}{(1+u)D^3}
    +\frac{u^iu_j}{D^3}.
\end{align}
Next, note
\begin{align}
&\sqrt{ \textit{det } g_{ij}  }= (1+u)^n \sqrt{\frac{|\nabla u|^2}{(1+u)^2} + 1}.
\end{align}
Therefore,
\begin{align}
    & \text{Area}(M) = \intsphere (1+u)^n \sqrt{\frac{|\nabla u|^2}{(1+u)^2} + 1} \textit{ dA}.
\end{align}We list a few more relevant formulae below:
\begin{align}
   & |\Omega| = \frac{1}{n+1}\intsphere (1 +u)^{n+1} \textit{dA},
   \label{volume_formula}
    \\
    &|\Omega \Delta B| =    \sum_{k=1}^{n+1} \int_{\sphere} \frac{1}{n+1}{n +1 \choose k}|u|^k \textit{dA},
    \label{symmetric_diff_formula}
    \\
    &\text{bar}(\Omega) = \frac{1}{\text{Area}(\partial B)}\intsphere (1+u)^{n+2} x \textit{dA}
    \label{bar_formula}.
\end{align}

Finally, we consider how to compute $\nabla_j [T_m]^j_i(D^2u)$ for nearly spherical sets. This is particularly useful when applying integration by parts on $I_k(\Omega)$ in Section 4. See \cite{MR3107515} for a similar computation.  We compute
\begin{align}
    \nabla_j [T_m]^{j}_i(D^2u) 
     &=
     \frac{1}{m!}  \nabla_j\delta^{jj_1j_2...j_{m}}_{ii_1i_2...i_{m}} u^{i_1}_{j_1}u^{i_2}_{j_2}\cdot \cdot\cdot u^{i_m}_{j_m}
   \nonumber  \\
     &=\frac{m}{m!} \delta^{jj_1j_2...j_{m}}_{ii_1i_2...i_{m}} 
      (\nabla_ju^{i_1}_{j_1})u^{i_2}_{j_2}\cdot \cdot\cdot u^{i_m}_{j_m}.
    \end{align}
    Note that
  \begin{align}
     \delta^{jj_1j_2...j_{m}}_{ii_1i_2...i_{m}} (\nabla_j u^{i_1}_{j_1})u^{i_2}_{j_2}\cdot \cdot\cdot u^{i_m}_{j_m}
     &=
      -\delta^{j_1jj_2...j_{m}}_{ii_1i_2...i_{m}} (\nabla_ju^{i_1}_{j_1})u^{i_2}_{j_2}\cdot \cdot\cdot u^{i_m}_{j_m}.
  \end{align}
  We obtain
    \begin{align}
       \nabla_j [T_m]^j_i(D^2u)    &=  
       \frac{1}{2(m-1)!}
          \delta^{jj_1j_2...j_{m}}_{ii_1i_2...i_{m}} 
          (\nabla_ju^{i_1}_{j_1} -\nabla_{j_1}u^{i_1}_{j})
          u^{i_2}_{j_2}
          \cdot \cdot\cdot u^{i_m}_{j_m}
     \nonumber  \\
    &=   \frac{1}{2(m-1)!}
    \delta^{jsj_1...j_{m-1}}_{ili_1...i_{m-1}}
    (u_pR^{pl}_{\;\;\;sj})
     u^{i_1}_{j_1}
          \cdot \cdot\cdot u^{i_{m-1}}_{j_{m-1}},
    \label{derivativenewton}
\end{align}
where $R^{pl}_{\;\;\;sj}$ is the curvature tensor on $\Omega$.
On $\mathbb{S}^n$, we know by the Gauss equation,
\begin{align}
    R^{pl}_{\;\;\;sj} &= h^p_{s}h^l_{j} - h^p_{j}h^l_{s}
=\delta^p_{s}\delta^l_{j} - \delta^p_{j}\delta^l_{s}.
\end{align}
Therefore,
 \begin{align}
       \nabla_j [T_m]^j_i(D^2u)   
    &=   
    \frac{1}{(m-1)!}\frac{1}{2}
    u_p(\delta^p_{s}\delta^l_{j} - \delta^p_{j}\delta^l_{s})
    \delta^{jsj_1...j_{m-1}}_{ili_1...i_{m-1}}
     u^{i_1}_{j_1}
          \cdot \cdot\cdot u^{i_{m-1}}_{j_{m-1}}
    \nonumber   
    \\
     &=   
    \frac{1}{(m-1)!}\frac{1}{2}
    \bigg (
    u_s
    \delta^{jsj_1...j_{m-1}}_{iji_1...i_{m-1}}
     u^{i_1}_{j_1}
          \cdot \cdot\cdot u^{i_{m-1}}_{j_{m-1}}
          -
       u_j
    \delta^{jlj_1...j_{m-1}}_{ili_1...i_{m-1}}
     u^{i_1}_{j_1}
          \cdot \cdot\cdot u^{i_{m-1}}_{j_{m-1}}
          \bigg)
    \nonumber  
    \\
        &=   
    \frac{-1}{(m-1)!}
       u_j
    \delta^{jlj_1...j_{m-1}}_{ili_1...i_{m-1}}
     u^{i_1}_{j_1}
          \cdot \cdot\cdot u^{i_{m-1}}_{j_{m-1}}
\nonumber
          \\
          &=
    -(n-m)
    u_j [T_{m-1}]^j_i(D^2u).
\end{align}

\section{Computation of $\sigma_k(L)$ for nearly spherical sets}

Our goal is to control a lower bound on the $(k,m)$-isoperimetric deficit $\delta_{k,m}(\Omega)$, where $M:= \partial \Omega$ is a nearly spherical set. In this section, we focus on calculating $\sigma_k(h^i_j)$, where

\begin{align}
    \sigma_k(h^i_j)
   & =
     \frac{1}{k!}
     \delta_{i_1\cdot\cdot\cdot i_k}^{j_1\cdot\cdot\cdot j_k}
    \bigg (\frac{-u^{i_1}_{j_1}}{D(1+u)}  
     + \frac{\delta^{i_1}_{j_1}}{D} 
     +  \frac{u_{j_1}^{s}u_su^{i_1}}{D^3(1+u)}
     + \frac{u^{i_1}u_{j_1}}{D^3} 
     \bigg )
      \nonumber
   \\
   &\cdot \cdot \cdot 
     \bigg(
     \frac{-u^{i_k}_{j_k}}{D(1+u)}  + \frac{\delta^{i_k}_{j_k}}{D} +  \frac{u_{j_k}^su_su^{i_k}}{D^3(1+u)}
     + \frac{u^{i_k}u_{j_k}}{D^3} 
     \bigg )
     \label{sigmaformula}.
\end{align}
In particular, the mean curvature on $M$ is given by
\begin{align}
    \sigma_1(h^i_j)
   &=
     \delta_{i_1}^{j_1}
    \bigg (\frac{-u^{i_1}_{j_1}}{D(1+u)}  
     + \frac{\delta^{i_1}_{j_1}}{D} +  \frac{u_{j_1}^su_su^{i_1}}{D^3(1+u)}
     + \frac{u^{i_1}u_{j_1}}{D^3} 
     \bigg )
     \nonumber 
     \\
     &= \frac{-\Delta u}{D(1+u)}  
     + \frac{n}{D} +  \frac{u_{i}^su_su^{i}}{D^3(1+u)}
     + \frac{|\nabla u|^2}{D^3}. 
\end{align}
Computing $\sigma_k(h^i_j)$ for any $k >0$ requires a bit more work, as we show in the following lemma.
\begin{lemma}
\label{sigmak}
Suppose $\Omega \subseteq R^{n+1}$ where $M = \{ (1+u(x))x: x \in \partial B  \}$ and $u \in C^2(\partial B)$. Then
\begin{align}
     \sigma_k(h^i_j)
=
\frac{1}{((1+u)^2 + |\nabla u|^2)^{\frac{k+2}{2}}}
 \sum_{m = 0}^{k}
 \frac{(-1)^m{n - m \choose k -m} }{(1+u)^m}
 \bigg((1+u)^2\sigma_m(D^2u)
  +
\frac{n+k-2m}{n-m} u^iu_j [T_m]^j_i(D^2u)
\bigg ).
\end{align}
\end{lemma}

\begin{proof}

In this proof we set
\begin{align} 
D :=\sqrt{(1+u)^2 + |\nabla u|^2} .
\end{align}
We expand out each term of $\sigma_k(h^i_j)$ in \eqref{sigmaformula}. Many of the terms in the expansion of this sum turn out to be zero. We compute the terms in the following cases:

\begin{enumerate}
    \item $m\geq 0$ instances of $\frac{-u^i_{j}}{D(1+u)}$, and the rest are in the form $\frac{\delta^i_j}{D}$.
    
First, consider the sum of all terms where $\frac{-u^i_{j}}{D(1+u)}$ occurs in the first $m$ terms. This equals
  \begin{align}
     & \frac{1}{k!}  \delta_{i_1\cdot\cdot\cdot i_k}^{j_1\cdot\cdot\cdot j_k}
     \frac{-u^{i_1}_{j_1}}{D(1+u)} 
     \cdot \cdot \cdot
     \frac{-u_{j_m}^{i_m}}{D(1+u)}
     \cdot \frac{\delta^{i_{m+1}}_{j_{m+1}}}{D}
     \cdot \cdot \cdot
     \frac{\delta^{i_k}_{j_k}}{D}
     \nonumber
     \\
     &= \frac{1}{k!}\ \frac{(-1)^m}{D^k(1+u)^{m}}  \delta_{i_1\cdot\cdot\cdot i_m}^{j_1\cdot\cdot\cdot j_m}
     u^{i_1}_{j_1}
     \cdot \cdot \cdot
     u^{i_m}_{j_m}
     \cdot {n-m \choose k-m} (k-m)!
        \nonumber
   \\
     &= \frac{{n-m \choose k-m}}{{k \choose m} }\frac{(-1)^m \sigma_m(D^2u) }{D^k(1+u)^{m}}.
     \label{unordered}
 \end{align}
 However,  to account for any permutation of the ordering of the terms above, we multiply \eqref{unordered} by ${k \choose m}$. 
 
 So, the sum of all terms in this case is:
 \begin{align}
     \frac{(-1)^m {n -m \choose k -m} \sigma_m(D^2u)}{D^k(1+u)^{m}} .
     \end{align}

\item One instance of $\frac{u^iu_j}{D^3}$ and $m\geq 1$  instances of $\frac{-u_{j}^i}{D(1+u)}$.

The sum of these terms is equal to:
\begin{align}
     &\frac{(-1)^m k \cdot {k -1 \choose m}}{k!}\delta_{i_1\cdot\cdot\cdot i_k}^{j_1\cdot\cdot\cdot j_k}
     \frac{u^{i_{1}}u_{j_{1}}}{D^3}
     \frac{u^{i_{2}}_{j_{2}}}{D(1+u)} 
     \cdot \cdot \cdot
     \frac{u_{j_{m+1}}^{i_{m+1}}}{D(1+u)}
     \cdot 
     \frac{\delta^{i_{m+2}}_{j_{m+2}}}{D}
     \cdot \cdot \cdot
     \frac{\delta^{i_k}_{j_k}}{D}
     \nonumber
\\
       &=\frac{(-1)^m{k -1 \choose m}}{(k-1)!} \frac{1}{{D^{k+2}(1+u)^{m}}}
      \delta_{i_1\cdot\cdot\cdot i_{m+1}}^{j_1\cdot\cdot\cdot j_{m+1}}
     u^{i_{1}}u_{j_{1}}
     u^{i_{2}}_{j_{2}}
     \cdot \cdot \cdot
    u_{j_{m+1}}^{i_{m+1}}
    { n - (m+1) \choose k - (m+1)}(k - (m+1))!
     \nonumber
     \\
       &={ n - (m+1) \choose k - (m+1)}\frac{(-1)^m}{D^{k+2}(1+u)^{m}}
         u^iu_j [T_m]^j_i(D^2u).
\end{align}

 \item One instance of $\frac{u_{j}^su_su^{i}}{D^3(1+u)}$ and $m\geq1$  instances of  $ \frac{u_{j}^{i}}{D(1+u)}$.
 
 In this case,  the sum of all the terms is:
 \begin{align}
    & \frac{(-1)^m {k \choose m}(k-m)}{k!}\delta_{i_1\cdot\cdot\cdot i_k}^{j_1\cdot\cdot\cdot j_k} \frac{u_{j_1}^su_{s}u^{i_1}}{D^3(1+u)}
   \frac{u_{j_2}^{i_2}}{D(1+u)}
   \cdot \cdot \cdot 
    \frac{u_{j_{m+1}}^{i_{m+1}}}{D(1+u)}
   \frac{\delta^{i_{m+2}}_{j_{m+2}}}{D} 
   \cdot \cdot \cdot
   \frac{\delta^{i_k}_{j_k}}{D}
   \nonumber
  \\
  &=  \frac{(-1)^m }{k!}\frac{{k \choose m}(k-m)}{D^{2 + k}(1+u)^{ m +1}}
  \delta_{i_1...i_{m+1}}^{j_1...j_{m+1}} 
   u_{j_1}^su_{s}u^{i_1}
   u_{j_2}^{i_2}
   \cdot \cdot \cdot 
   u_{j_{m+1}}^{i_{m+1}}
   \cdot 
   {n - (m+1) \choose k - (m+1)}(k -(m+1))!
    \nonumber
      \\
  &=  {n - (m+1) \choose k - (m+1)} \frac{(-1)^m}{D^{2 + k}(1+u)^{ m + 1}}u_{j}^su_{s}u^{i}
 [T_m]^j_i(D^2u).
\end{align}

\item When there  are either two instances of $\frac{u^iu_j}{D^3(1+u)^3}$,  two instances of $\frac{u_{js}u^su^{i}}{D^3(1+u)^4}$, or one instance of $\frac{u^iu_j}{D^3(1+u)^3}$ and one instance of $\frac{u_{js}u^su^{i}}{D^3(1+u)^4}$.

In this case, we apply the following Lemma \ref{zerolem} to conclude the sum of all these terms is zero.

\end{enumerate}

Next, we simplify the expression for $\sigma_k(h^i_j)$ by noting the identity
\begin{align}
    u_{j}^s[T_m]^j_i(D^2u) = \delta^{s}_i\sigma_{m+1}(D^2u)  - [T_{m+1}]^{s}_i(D^2u). 
    \end{align}
We compute
 \begin{align}
    \sigma_k(h^i_j) = 
   & \sum_{m = 0}^k\frac{ {n -m\choose k -m}(-1)^m \sigma_m(D^2u)}{D^k(1+u)^{m}}
  +\sum_{m=0}^{k-1}{ n - (m+1) \choose k - (m+1)}\frac{(-1)^m  u^iu_j [T_m]^j_i(D^2u)}{D^{k+2}(1+u)^{m}}
 \nonumber
 \\
         &+  \sum_{m = 0}^{k-1}{n - (m+1) \choose k - (m+1)} \frac{(-1)^mu_{j}^su_{s}u^{i}
 [T_m]^j_i(D^2u)}{D^{k+2}(1+u)^{ m +1}}
\nonumber
\\
   &= \sum_{m = 0}^k\frac{ {n -m\choose k -m}(-1)^m \sigma_m(D^2u)}{D^k(1+u)^{m}}
  +\sum_{m=0}^{k-1}{ n - (m+1) \choose k - (m+1)}\frac{(-1)^m  u^iu_j [T_m]^j_i(D^2u)}{D^{k+2}(1+u)^{m }}
\nonumber
\\
         &+  \sum_{m = 0}^{k-1}{n - (m+1) \choose k - (m+1)} \frac{(-1)^m(-u^{i}u_{j}[T_{m+1}]^{j}_i(D^2u) + |\nabla u|^2\sigma_{m+1}(D^2u))}{D^{k+2}(1+u)^{ m + 1}}
\nonumber
\\
   &= \sum_{m = 0}^k\frac{[(1+u)^2 +|\nabla u|^2] {n -m\choose k -m}(-1)^m \sigma_m(D^2u)}{D^{k+2}(1+u)^{m}}
  +\sum_{m=0}^{k-1}{ n - (m+1) \choose k - (m+1)}\frac{(-1)^m  u^iu_j [T_m]^j_i(D^2u)}{D^{k+2}(1+u)^{m }}
  \nonumber
  \\
&+  \sum_{m = 1}^{k}{n - m \choose k - m} \frac{(-1)^{m}
(u^{i}u_{j}[T_m]^j_i(D^2u) - |\nabla u|^2\sigma_{m}(D^2u))}{D^{k+2}(1+u)^{ m}}
\nonumber
\\
   &=
   \frac{1}{D^{k+2}}
 \sum_{m = 0}^{k}
 \frac{(-1)^m{n - m \choose k -m} }{(1+u)^m}
 \bigg((1+u)^2\sigma_m(D^2u)
  +
\frac{n+k-2m}{n-m} u^iu_j [T_m]^j_i(D^2u)
\bigg ).
 \end{align}
\end{proof}

Now we give a quick proof of the lemma that was used in case 4 in Lemma \ref{sigmak}.
\begin{lemma}
\label{zerolem}
 Suppose $2 \leq k \leq n$ where $M_1,..., M_{k-2}$  are $n \times n$ matrices, and $w_1, w_2$, $v$ are $n$-dimensional vectors. Then, 
 \begin{align}
     \Sigma_k (w_1v^T, w_2v^T, M_1, M_2, M_3,..., M_{k-2}) =0. 
      \end{align}
\end{lemma}

\begin{proof}
We compute,
\begin{align}
    \Sigma_k (w_1v^T, w_2v^T, M_1, M_2,..., M_{k-2}) 
    &= \frac{1}{(k-1)!}\delta_{i_1i_2...i_k}^{j_1j_2...j_k}
     w_1^{i_1}v_{j_1}
      w_2^{i_2}v_{j_2}
      {M_1}^{i_3}_{j_3}
      \cdot \cdot \cdot 
       {M_{k-2}}^{i_k}_{j_k}
       \nonumber
       \\
         &=
         \frac{1}{(k-1)!}\delta_{i_1i_2...i_k}^{j_1j_2...j_k}
     w_1^{i_1}v_{j_2}
      w_2^{i_2}v_{j_1}
      {M_1}^{i_3}_{j_3}
      \cdot \cdot \cdot 
       {M_{k-2}}^{i_k}_{j_k}
       \nonumber
     \\
     &=
        - \frac{1}{(k-1)!}\delta_{i_1i_2...i_k}^{j_2j_1...j_k}
     w_1^{i_1}v_{j_2}
      w_2^{i_2}v_{j_1}
      {M_1}^{i_3}_{j_3}
      \cdot \cdot \cdot 
       {M_{k-2}}^{i_k}_{j_k}
       \nonumber
       \\
       &= 
       - \Sigma_k (w_1v^T, w_2v^T, M_1, M_2,..., M_{k-2}).
\end{align}
Since 
     $\Sigma_k (w_1v^T, w_2v^T, M_1, M_2,..., M_{k-2}) = -  \Sigma_k (w_1v^T, w_2v^T, M_1, M_2,..., M_{k-2}) $,
we conclude
 \begin{align}
     \Sigma_k (w_1v^T, w_2v^T, M_1, M_2, M_3,..., M_{k-2}) =0. 
      \end{align}
\end{proof}

\section{$I_k(\Omega)$ for nearly spherical sets}
We continue to look at $I_k(\Omega)$ for starshaped domains, but now we add in the additional assumption that $||u||_{W^{2,\infty}}< \epsilon$,   making $M:=\partial \Omega$ a \textit{nearly spherical set} as described in Definition \ref{nearlysphere}.
Note,
  \begin{align}
      \int_{M} \sigma_k(h_j^i) d\mu =\intsphere \sigma_k(h_j^i)(1+u)^n\sqrt{1 + \frac{|\nabla u|^2}{(1+u)^2}} \textit{ dA}. 
        \end{align}
  Using our formula for $\sigma_k(h_j^i)$ in Lemma \ref{sigmak}, we have the following expression for $ \int_{M} \sigma_k(h_j^i) d\mu $:
  \begin{align}
\intsphere  \sum_{m = 0}^{k}  \frac{ (-1)^m{n - m \choose k -m}(1+u)^{n-m-1} }{((1+u)^2 + |\nabla u|^2)^{\frac{k+1}{2}}}
 \bigg((1+u)^2\sigma_m(D^2u)
  +
\frac{n+k-2m}{n-m} u^iu_j [T_m]^j_i(D^2u)
\bigg )\textit{dA}.
\end{align}

 Later, in the main theorems, our analysis deals mainly with lower order terms in $O(|u|)$ and $O(|\nabla u|)$. In the following lemma, we expand out $\frac{1}{((1+u)^2 + |\nabla u|^2)^{(k+1)/2}}$ in the integral using its Taylor expansion, and we are able to group all of the higher order terms in $O(\epsilon)||u||_{L^2}^2 + O(\epsilon)||\nabla u||_{L^2}^2$.

\begin{lemma}
\label{taylor}
Suppose $u\in C^1( \partial B)$ and for sufficiently small $\epsilon>0$ that $||u||_{L^{\infty}}, ||\nabla u||_{L^{\infty}} < \epsilon$. Then, 
\begin{align}
    \frac{1}{(|\nabla u|^2 + (1+u)^2)^{\frac{m}{2}}}
      &=  1 -mu + \frac{m(m+1)}{2}u^2 -\frac{m}{2}|\nabla u|^2
    +O(\epsilon)u^2 + O(\epsilon)|\nabla u|^2.
\end{align}
\end{lemma}

\begin{proof}

First, we will expand out
$(|\nabla u|^2 + (1+u)^2)^{\frac{1}{2}})^{-1/2}$, which we rewrite as $( 1 + 2u + u^2 + |\nabla u|^2 )^{-1/2}$.
Setting $f(x) =( 1 +x)^{-1/2} $, we obtain
\begin{align}
    f^{n}(0) = \frac{(-1)^n\prod_{m = 1}^n (2m - 1)}{2^n}.
\end{align}
In the radius of convergence for its Taylor expansion, 
$f(x) = \sum_{n=0}^{\infty} c_n x^n $,
where $c_n =  \frac{(-1)^n\prod_{m = 1}^n (2m - 1)}{n!2^n}$ for $n \geq 1$ and $c_0 = 1$.
Note that
\begin{align}
    |c_n| &= \frac{\prod_{m = 1}^n (2m - 1)}{n!2^n}
    =\frac{(2n)!}{(n!)^2 4^n}
   \leq 1.
\end{align}
Then, for small $|u|$ and $|\nabla u|$, and setting $g(u) := 2u + u^2 + |\nabla u|^2$,
\begin{align}
\frac{1}{(|\nabla u|^2 + (1+u)^2)^{\frac{1}{2}}}
    &= 1 - \frac{1}{2}(g(u)) + \frac{3}{8}(g(u))^2
    + \sum_{n =3 }^{\infty}c_n (g(u))^n.
\end{align}
Furthermore, for $|g(u)| \leq \frac{1}{2}$,
\begin{align}
 \bigg | \sum_{n =3 }^{\infty}c_n (g(u))^n \bigg |
   & \leq \sum_{n =3 }^{\infty} |g(u)|^n
    =\frac{|g(u)|^3}{1 - g(u)}
  \leq 2|2u + u^2 + |\nabla u|^2|^3
 =  O(\epsilon)u^2 + O(\epsilon)|\nabla u|^2.
\end{align}
Therefore, 
\begin{align}
  \frac{1}{(|\nabla u|^2 + (1+u)^2)^{\frac{1}{2}}}
    &= 1 - \frac{1}{2}(2u + u^2 + |\nabla u|^2)+ \frac{3}{8}(2u + u^2 + |\nabla u|^2)^2
    + \sum_{n =3 }^{\infty}c_n (g(u))^n
    \nonumber
      \\
    &=1 -u + u^2 -\frac{1}{2}|\nabla u|^2 + O(\epsilon)u^2 + O(\epsilon)|\nabla u|^2.
\end{align}
We conclude,
\begin{align}
    \frac{1}{(|\nabla u|^2 + (1+u)^2)^{\frac{m}{2}}}
    &=
    \bigg (1 + (-u + u^2 -\frac{1}{2}|\nabla u|^2) + O(\epsilon)u^2 + O(\epsilon)|\nabla u|^2 \bigg)^m
    \nonumber
    \\
    &= \sum_{j= 0}^m {m \choose j}(-u + u^2 -\frac{1}{2}|\nabla u|^2)^j
    +O(\epsilon)u^2 + O(\epsilon)|\nabla u|^2 
   \nonumber
   \\
      &= 1 -mu + \frac{m(m+1)}{2}u^2 -\frac{m}{2}|\nabla u|^2
    +O(\epsilon)u^2 + O(\epsilon)|\nabla u|^2.
\end{align}
\end{proof}

Using the expansion in Lemma \ref{taylor}, we further expand $\int_{M} \sigma_k(h_j^i) \textit{ d}\mu$. We perform integration by parts on many of the terms to convert them to include $|\nabla u|^2$. The new format of the integral will be useful later on when we find a lower bound involving the Fraenkel asymmetry.

\begin{lemma}
\label{expansion}
Suppose $\Omega= \{ (1+u(\frac{x}{|x|}))x: x \in B  \} \subseteq \mathbb{R}^{n+1}$ and $u \in C^3(\partial B)$.
If $||u||_{W^{2,\infty}}< \epsilon$, then
 \begin{align}
      \int_{M} \sigma_k(h_j^i) d\mu
    &=
    \intsphere 
  {n \choose k} +{n \choose k}  ( n -k)u+{n \choose k}\frac{(n - k)(n-k-1)}{2}u^2
\nonumber
\\
&+\sum_{m = 0}^{k}  (-1)^{m}{n - m \choose k - m}\frac{(n-k)(k+1)}{2(m+1)(n-m)} |\nabla u|^2 \sigma_{m}(D^2u)
\textit{ dA} 
+  O(\epsilon)||\nabla u||^2_{L^2} +  O(\epsilon)||u||^2_{L^2}.
 \end{align}
\end{lemma}

\begin{proof}

Applying the Taylor expansion in Lemma \ref{taylor}, we find
       \begin{align}
    \int_{M} \sigma_k(h^i_j) d\mu 
    &=
    \sum_{m = 0}^{k} 
     (-1)^m{n - m \choose k -m}
     \intsphere 
 \bigg (1+ ( n -m -k)u + \frac{(n-m-k)(n-m-k -1)}{2}u^2
\nonumber 
\\
&-
\frac{k+1}{2}|\nabla u|^2
\bigg )
 \sigma_m(D^2u)
+
\frac{n+k-2m}{n-m} u^iu_j [T_m]^j_i(D^2u)
\textit{dA} 
+ O(\epsilon)||u||_{L^2}^2 +  O(\epsilon)||\nabla u||_{L^2}^2.
    \end{align}
 Next, recall from the preliminaries that   
    \begin{align}
        \nabla_j [T_m]^j_i(D^2u) = - (n-m)u_j [T_{m-1}]^j_i(D^2u),
    \end{align}
    and
    \begin{align}
        \sigma_m(D^2u) = \frac{1}{m}u_{j}^i[T_{m-1}]^j_i(D^2u). 
    \end{align}
Using these identities, we rewrite many of the terms using integration by parts.
First, for $m\geq 1$
 \begin{align}
     \intsphere |\nabla u|^2\sigma_{m}(D^2u)  \textit{dA} 
     &=
          \frac{1}{m}\intsphere |\nabla u|^2u_{j}^i[T_{m-1}]^j_i(D^2u) \textit{dA} 
         \nonumber 
         \\
           &=
          \frac{-1}{m}\intsphere u^i2u_su^s_{j}[T_{m-1}]^j_i(D^2u)+u^i|\nabla u|^2 \nabla_j [T_{m-1}]^j_i(D^2u) \textit{dA} 
        \nonumber   
          \\
           &=
          \frac{-2}{m}\intsphere u^iu_su^s_{j}[T_{m-1}]^j_i(D^2u)\textit{dA} 
          +  O(\epsilon)||\nabla u||^2_{L^2}
          \nonumber 
          \\
           &=
          \frac{2}{m}\intsphere u^iu_j[T_m]^j_i(D^2u) - |\nabla u|^2\sigma_{m}(D^2u) \textit{dA} 
          +  O(\epsilon)||\nabla u||^2_{L^2}.
 \end{align}
 Therefore,
 \begin{align}
 \label{int_by_parts_identity}
       \intsphere u^iu_j[T_m]^j_i(D^2u)  \textit{dA}
 &= \frac{m+2}{2}\intsphere |\nabla u|^2\sigma_{m}(D^2u) \textit{dA} 
          +  O(\epsilon)||\nabla u||^2_{L^2}.
 \end{align}
Next, we integrate each $\sigma_{m}(D^2u)$ term. For $m \geq 2$, we have
\begin{align}
    \intsphere \sigma_{m}(D^2u) \textit{dA}
    &= 
    \frac{1}{m} \intsphere u^i_{j}[T_{m-1}]^j_i(D^2u) \textit{dA}
  \nonumber  \\
      &= 
    \frac{-1}{m} \intsphere u^i \nabla_j[T_{m-1}]^j_i(D^2u) \textit{dA}
  \nonumber
  \\
      &= 
    \frac{n-m+1}{m} \intsphere u^{i}u_j
    [T_{m-2}]^j_i(D^2u) \textit{dA}.
    \end{align}
  By \eqref{int_by_parts_identity},
  \begin{align}
      \intsphere \sigma_{m}(D^2u) \textit{dA}   &= 
    \frac{n-m+1}{2}\intsphere  |\nabla u|^2\sigma_{m-2}(D^2u)\textit{dA} +  O(\epsilon)||\nabla u||^2_{L^2}.
\end{align} 
Also, for $m=0$ and $1$,
\begin{align}
     \intsphere \sigma_{0}(D^2u) \textit{dA}  &=  \intsphere 1 \textit{dA}
\text{  and  }
       \intsphere \sigma_{1}(D^2u) \textit{dA} =  0.
\end{align}
Similarly, for $m\geq 1$, 
\begin{align}
    \intsphere u\sigma_{m}(D^2u) \textit{dA}
    &= 
    \frac{1}{m} \intsphere uu^i_{j}[T_{m-1}]^j_i(D^2u) \textit{dA}
 \nonumber 
 \\
      &= 
    \frac{-1}{m} \intsphere u^iu_j[T_{m-1}]^j_i(D^2u)+uu^i \nabla_j[T_{m-1}]^j_i(D^2u) \textit{dA}
      \nonumber 
      \\
      &= 
    \frac{-1}{m} \intsphere u^iu_j[T_{m-1}]^j_i(D^2u) \textit{dA}
    + 
    O(\epsilon)||\nabla u||^2_{L^2}.
    \end{align}
    By \eqref{int_by_parts_identity},
    \begin{align}
 \intsphere u\sigma_{m}(D^2u) \textit{dA}     &= 
    \frac{-(m+1)}{2m} \intsphere  |\nabla u|^2\sigma_{m-1}(D^2u) \textit{dA}
    + 
    O(\epsilon)||\nabla u||^2_{L^2}.
\end{align}  
And,
\begin{align}
    \intsphere u\sigma_{0}(D^2u) \textit{dA}
    &= 
\intsphere u   \textit{dA}.
\end{align}  
  Lastly, for $m\geq 1$, 
  \begin{align}
    \intsphere u^2\sigma_{m}(D^2u) \textit{dA}
    &= 
    \frac{1}{m} \intsphere u^2u^i_{j}[T_{m-1}]^j_i(D^2u) \textit{dA}
    \nonumber 
    \\
      &= 
    \frac{-1}{m} \intsphere u^i
    \bigg (
    2uu_j[T_{m-1}]^j_i(D^2u)+u^2 \nabla_j[T_{m-1}]^j_i(D^2u)
    \bigg )
    \textit{dA}
    \nonumber 
    \\
    &=   O(\epsilon)||\nabla u||^2_{L^2}.
\end{align}   
And,
\begin{align}
    \intsphere u^2\sigma_{0}(D^2u) \textit{dA}
    = \intsphere u^2\textit{dA}.
\end{align}
All together, we find
\begin{align}
 \int_M \sigma_{k}(h^i_j) d\mu
    &=\intsphere 
  {n \choose k} +{n \choose k}  ( n -k)u+{n \choose k}\frac{(n - k)(n-k-1)}{2}u^2
 \nonumber 
 \\
&+\sum_{m = 0}^{k}  (-1)^{m}{n - m \choose k - m}\frac{(n-k)(k+1)}{2(m+1)(n-m)} |\nabla u|^2 \sigma_{m}(D^2u)
\textit{ dA} 
+  O(\epsilon)||\nabla u||^2_{L^2} +  O(\epsilon)||u||^2_{L^2}.
\end{align}
\end{proof}

Next we turn our attention to the $(k,\minus 1)$-isoperimetric deficit. Recall from Section \ref{deficitprelim},
\begin{align}
\delta_{k,\minus 1}(\Omega) = \frac{I_k(\Omega) - I_k(B_{\Omega,\minus 1})}{I_k(B_{\Omega,\minus 1})} 
\end{align}
where $B_{\Omega,\minus 1}$ is the ball centered at the origin satisfying $\text{Vol}(B_{\Omega,\minus 1}) = \text{Vol}(\Omega)$. In particular, if $\Omega$ is normalized so that $\text{Vol}(\Omega) = \text{Vol}(B)$, then
\begin{align}
\delta_{k,\minus1}(\Omega) = \frac{I_k(\Omega) - I_k(B)}{I_k(B)} .
\end{align}
In the next proposition, with the additional assumption that the barycenter of $\Omega$ is at the origin, we are able to bound $I_k(\Omega) - I_k(B)$ below by terms involving $||u||_{W^{1,2}}$. In order to get our main theorem bounding $\delta_{k,\minus1}(\Omega)$ below by $\alpha^2(\Omega)$ (see Section \ref{deficitprelim}), we only need the term with $||u||_{L^2}^2$ in the lower bound. However, we form a stronger statement that also includes  $||\nabla u||_{L^2}^2$ in the lower bound.
\begin{proposition}
\label{propvol}
 Suppose  $\Omega= \{ (1+u(\frac{x}{|x|}))x: x \in B  \} \subseteq \mathbb{R}^{n+1}$ where $u \in C^3(\partial B)$, $\text{Vol}(\Omega) = \text{Vol}(B)$, and $\text{bar}(\Omega) = 0$.
Additionally, assume for sufficiently small $\epsilon >0$ that $||u||_{W^{2,\infty}}< \epsilon$. Then,
\begin{align}
     I_k(\Omega) - I_k(B) \geq 
     {n\choose k}\frac{(n-k)(k+1)}{2n}
    \bigg (
     \bigg(1+O(\epsilon)
     \bigg)||u ||_{L^2}^2
     +
   \bigg (\frac{1}{2} + O(\epsilon)
    \bigg )
    ||\nabla u ||_{L^2}^2
     \bigg).
\end{align}
\end{proposition}
\begin{proof}

From Lemma \ref{expansion}, 
\begin{align}
\label{ik_minus_ib}
     I_k(\Omega) - I_k(B)
     &=
       \intsphere 
  {n \choose k}  ( n -k)u+{n \choose k}\frac{(n - k)(n-k-1)}{2}u^2
\nonumber
\\
&+\sum_{m = 0}^{k}  (-1)^{m}{n - m \choose k - m}\frac{(n-k)(k+1)}{2(m+1)(n-m)} |\nabla u|^2 \sigma_{m}(D^2u)
\textit{ dA} 
+  O(\epsilon)||\nabla u||^2_{L^2} +  O(\epsilon)||u||^2_{L^2}.
\end{align}
Using the assumption that $\text{Vol}(\Omega) = \text{Vol}(B)$, we have from formula \eqref{volume_formula} for the volume that
\begin{align}
\intsphere  u \textit{ dA} = \intsphere \frac{-n}{2}u^2 \textit{ dA} +  O(\epsilon)||u||^2_{L^2}. 
\end{align}
Substituting this expression into \eqref{ik_minus_ib} yields
\begin{align}
    I_k(\Omega) - I_k(B)
&=\intsphere 
{n\choose k}\frac{(n-k)(k+1)}{2n} |\nabla u|^2
-
{n \choose k}\frac{(n - k)(k+1)}{2}u^2
\nonumber
\\
&+\sum_{m = 1}^{k}  (-1)^{m}{n - m \choose k - m}\frac{(n-k)(k+1)}{2(m+1)(n-m)} |\nabla u|^2 \sigma_{m}(D^2u)
\textit{dA} 
+  O(\epsilon)||\nabla u||^2_{L^2} +  O(\epsilon)||u||^2_{L^2}.
\label{Ikbeforespherical}
 \end{align}
 Then, using the assumptions that $\text{Vol}(\Omega) = \text{Vol}(B)$ and $\text{bar}(\Omega) = 0$, Cicalese and Leonardi showed  (see Lemma 4.2 in \cite{MR2980529}), by writing $u$ in terms of its spherical harmonics basis, that
 \begin{align}
 \label{inequality_spherical_volume}
      ||\nabla u||_{L^2}^2 \geq 2(n+1) ||u ||_{L^2}^2 + O(\epsilon)||u||^2_{L^2}.
      \end{align}
Finally, by applying the inequality \eqref{inequality_spherical_volume} to \eqref{Ikbeforespherical}, we find that $I_k(\Omega) - I_k(B)$ is bounded below by the following expression:
 \begin{align}
 & {n\choose k}\frac{(n-k)(k+1)}{2n}
\bigg(
\frac{1}{2}||\nabla u||^2_{L^2}
+
(n+1)||u ||_{L^2}^2 
-n||u ||_{L^2}^2 
\bigg)
+  O(\epsilon)||\nabla u||^2_{L^2} +  O(\epsilon)||u||^2_{L^2}.
 \end{align}
\end{proof}

Now, with the lower bound on $I_k(\Omega) - I_k(B)$ being controlled by $||u||^2_{L^2}$, we are equipped to show one of our main results.
Observe, as shown in \cite{MR2980529}, that for $||u||_{L^{\infty}} < \epsilon$,  H\"{o}lder's inequality yields
\begin{align}
    \frac{|\Omega \Delta B|}{|B|} &= 
    \frac{1}{|B|}
    \bigg (||u||_{L^1} + \sum_{k=2}^{n+1} \int_{\sphere} \frac{1}{n+1}{n +1 \choose k}|u|^k \textit{ dA} \bigg ) 
    \nonumber
    \\
   & \leq \frac{1}{|B|}\bigg(\text{Area}(\partial B)^{1/2}||u||_{L^2} 
    + \sum_{k=2}^{n+1} \int_{\sphere} \frac{1}{n+1}{n +1 \choose k}|u|^k \textit{ dA} \bigg  ). 
\end{align}
Therefore, when $\text{Vol}(\Omega) = \text{Vol}(B)$,
\begin{align}
\label{holder_frankael}
    \alpha^2(\Omega)
    \leq  \frac{|\Omega \Delta B|^2}{|B|^2}
    \leq
    \frac{\text{Area}(\partial B)}{|B|^2}||u||^2_{L^2} 
    + O(\epsilon)||u||_{L^2}^2
    =
    \frac{(n+1)^2}{\text{Area}(\partial B)}||u||^2_{L^2} 
    + O(\epsilon)||u||_{L^2}^2.
\end{align}
Now we are ready to prove Theorem \ref{firstthm}.

\vspace{.07in}
\noindent \textbf{Theorem \ref{firstthm}.}
\textit{ Suppose  $\Omega= \{ (1+u(\frac{x}{|x|}))x: x \in B  \} \subseteq \mathbb{R}^{n+1}$, where $u \in C^3(\partial B)$, $\text{Vol}(\Omega) = \text{Vol}(B)$, and $\text{bar}(\Omega) =0$.}
 \textit{For all $\eta >0$, there exists $\epsilon >0$ such that if $||u||_{W^{2,\infty}}< \epsilon$, then}
\begin{align}
    \delta_{k,j}(\Omega)
    \geq 
\bigg (
\frac{(n-k)(k+1)}{2n(n+1)^2}  -\eta \bigg)\alpha^2(\Omega).
\end{align}

\begin{proof}
Suppose $||u||_{W^{2,\infty}}< \epsilon$. Since $\text{Vol}(\Omega) = \text{Vol}(B)$, the definition for the $(k,\minus 1)$-isoperimetric deficit becomes
\begin{align} \delta_{k,\minus 1}(\Omega) = \frac{I_k(\Omega) - I_k(B)}{I_k(B)} =\frac{I_k(\Omega) - I_k(B)}{{n \choose k}\text{Area}(\partial B)} .
  \end{align} 
From Proposition \ref{propvol},  we have
\begin{align} \delta_{k,\minus 1}(\Omega) \geq \frac{(n-k)(k+1)}{2n\text{Area}(\partial B)}||u ||_{L^2}^2 + O(\epsilon) ||u ||_{L^2}^2. 
\end{align}
Next, as noted in \eqref{holder_frankael},
  \begin{align}
     \alpha^2(\Omega) \leq \frac{|\Omega \Delta B |^2 }{|B|^2}  \leq \frac{(n+1)^2}{
     \text{Area}(\partial B)} ||u||_{L^2}^2 + O(\epsilon)||u||_{L^2}^2.  
       \end{align} 
It follows that
       \begin{align}
     ||u||_{L^2}^2 
      &\geq 
       \frac{\text{Area}(\partial B)}{(n+1)^2} \alpha^2(\Omega)+ O(\epsilon)\alpha^2(\Omega).
       \label{u2normbound}
         \end{align} 
Therefore,
 \begin{align}\delta_{k,\minus 1}(\Omega) 
 &\geq 
 \frac{(n-k)(k+1)}{2n\text{Area}(\partial B)}
 (1 + O(\epsilon)) ||u ||_{L^2}^2
 \nonumber
 \\
 &\geq 
 \frac{(n-k)(k+1)}{2n(n+1)^2}\alpha^2(\Omega) 
   + O(\epsilon)\alpha^2(\Omega). 
   \end{align}

\end{proof}

\section{Quantitative isoperimetric inequality for $\delta_{k,j}(\Omega)$ when $j\geq 0$}
In this section we extend the result from Theorem \ref{firstthm} to $\delta_{k,j}(\Omega)$ for $0 \leq j < k$. The proof turns out to be quite similar to the case for $\delta_{k,\minus 1}(\Omega)$ in the previous section. The expression for $I_k(\Omega)$ will contain the quantity $C\intsphere |\nabla u|^2 - nu^2 \textit{dA}$ under the assumption $I_j(\Omega) = I_j(B)$, which we bound in the same manner as in Proposition \ref{propvol}.

We begin with the following proposition.

\begin{proposition}
\label{propgen}
 Fix  $j$ where $ 0 \leq j < k$. Suppose $\Omega = \{ (1+u(\frac{x}{|x|}))x: x \in B \} \subseteq \mathbb{R}^{n+1}$, where $u \in C^3(\partial B)$, $I_j(\Omega) = I_j(B)$, and $\text{bar}(\Omega) = 0$.
Assume for sufficiently small $\epsilon >0$ that $||u||_{W^{2,\infty}}<\epsilon$. Then,
\begin{align} I_k(\Omega) - I_k(B) \geq
 {n  \choose k}\frac{(n-k)(k-j)}{2n}
 \bigg( 
 \bigg(1+ O(\epsilon) \bigg) || u||_{L^2}^2 
 + \bigg(\frac{1}{2} +O(\epsilon)
 \bigg )||\nabla u||_{L^2}^2 
 \bigg ).
   \end{align} 
\end{proposition}
\begin{proof}

First,
for any $s \geq 0$, we have from Lemma \ref{expansion} that
 \begin{align}
 \label{ik_sub_later}
    I_s(\Omega) - I_s(B) 
    &=
    \intsphere 
  {n \choose s}  ( n -s)u+{n \choose s}\frac{(n - s)(n-s-1)}{2}u^2
\nonumber
\\
&+\sum_{m = 0}^{s}  (-1)^{m}{n - m \choose s - m}\frac{(n-s)(s+1)}{2(m+1)(n-m)} |\nabla u|^2 \sigma_{m}(D^2u)
\textit{dA} 
+  O(\epsilon)||\nabla u||^2_{L^2} +  O(\epsilon)||u||^2_{L^2}.
 \end{align}
 Therefore, if $I_j(\Omega) = I_j(B)$, 
 \begin{align}
 \label{ij_equals_ball}
    \intsphere
    u
    \textit{dA} 
    &=
    -\intsphere 
 \frac{n-j-1}{2}u^2
 +
 \frac{j+1}{2n} |\nabla u|^2
+
\bigg(
\frac{1}{{n \choose j}}
\sum_{m = 1}^{k} 
(-1)^m{n -m \choose j-m}
\frac{j+1}{2(m+1)(n-m)} \sigma_{m}(D^2u) \bigg) 
\textit{dA} 
\nonumber
\\
&+  O(\epsilon)||\nabla u||^2_{L^2} +  O(\epsilon)||u||^2_{L^2}.
 \end{align}
Substituting this expression in \eqref{ik_sub_later} for $s=k$ yields
 \begin{align}
    I_k(\Omega) - I_k(B) 
    &=
  {n  \choose k}\frac{(n-k)(k-j)}{2n}  \intsphere 
 |\nabla u|^2 - nu^2
+
\bigg(
\sum_{m = 1}^{k} 
d_m\sigma_{m}(D^2u) \bigg) |\nabla u|^2 
\textit{dA}
\nonumber
\\
&+  O(\epsilon)||\nabla u||^2_{L^2} +  O(\epsilon)||u||^2_{L^2},
 \end{align}
 where each $d_m$ is coefficient for $\sigma_m(D^2u)|\nabla u|^2$.
 Using the assumptions that $I_j(\Omega) = I_j(B)$ and $\text{bar}(\Omega) =0$, we show in the following lemma that
\begin{align}
    ||\nabla u||_{L^2}^2 \geq 2(n+1) ||u ||_{L^2}^2 + O(\epsilon^2)||u||_{L^2}^2 +O(\epsilon^2)||\nabla u||^2_{L^2}.
    \end{align} 
Therefore,
\begin{align}
    I_k(\Omega) - I_k(B)
&\geq
{n  \choose k}\frac{(n-k)(k-j)}{2n}
\bigg(
\frac{1}{2}||\nabla u||^2_{L^2} + (n+1)||u||_{L^2}^2 -n ||u||_{L^2}^2
\bigg)
\nonumber
\\
& +\intsphere 
\bigg(
\sum_{m = 1}^{k} 
d_m \sigma_{m}(D^2u) \bigg) |\nabla u|^2 
\textit{dA} 
+  O(\epsilon) ||\nabla u||^2_{L^2} +  O(\epsilon)||u||^2_{L^2}
\nonumber
\\
&\geq
{n  \choose k}\frac{(n-k)(k-j)}{4n}||\nabla u||^2_{L^2} + {n  \choose k}\frac{(n-k)(k-j)}{2n}||u||_{L^2}^2
+  O(\epsilon)||\nabla u||^2_{L^2} +  O(\epsilon)||u||^2_{L^2}.
 \end{align}

\end{proof}

We now prove the lower bound on $||\nabla u||_{L^2}^2$ used in the previous lemma. The proof closely resembles that in \cite{MR2980529} by Cicalese and Leonardi in their work with the classical quantitative isoperimetric inequality (see also \cite{MR859955} and \cite{ MR942426} by Fuglede). 

\begin{lemma}
\label{spherharmbound}
Suppose  $\Omega = \{ (1+u(\frac{x}{|x|}))x: x \in B \} \subseteq \mathbb{R}^{n+1}$, with $u \in C^3(\partial B)$, $\text{bar}(\Omega) = 0$, and $I_j(\Omega) = I_j(B)$ for a fixed $j$ where $0 \leq j \leq n$. It holds that
\begin{align}
 || \nabla u||^2_{L^2}
    \geq
    2(n+1)||u||_{L^2}^2 + O(\epsilon^2) ||u||_{L^2}^2+ O(\epsilon^2)||\nabla u||_{L^2}^2.
\end{align}
\end{lemma}
\begin{proof}
We write
\begin{align}
    u = \sum_{k=0}^{\infty} a_kY_k,
\end{align}
where $\{Y_k \}$ are spherical harmonics which form an orthonormal basis for $L^2(\partial B)$. Since $Y_0 =1$ we have
\begin{align}
    a_0 = <u, 1>_{L^2(\partial B)} = \intsphere u \textit{ dA}.
\end{align}
Additionally, using the assumption $I_j(\Omega) = I_j(B)$, we have from  \eqref{ij_equals_ball} that
 \begin{align}
    \intsphere
    u
    \textit{dA} 
    &=
    -\intsphere 
 \frac{n-j-1}{2}u^2
 +
 \frac{j+1}{2n} |\nabla u|^2
+
\bigg(
\frac{1}{{n \choose j}}
\sum_{m = 1}^{k} 
(-1)^m{n -m \choose j-m}
\frac{j+1}{2(m+1)(n-m)} \sigma_{m}(D^2u) \bigg) 
\textit{dA} 
\nonumber
\\
&+  O(\epsilon)||\nabla u||^2_{L^2} +  O(\epsilon)||u||^2_{L^2}.
 \end{align}
This further implies that $\intsphere u
    \textit{dA} = O(\epsilon^2)$. Hence,
    \begin{align}
        a_0^2 
    &= O(\epsilon^2)||u||_{L^2}^2 +  O(\epsilon^2)||\nabla u||^2_{L^2}.
    \end{align}As shown in \cite{MR2980529}, combining  $\text{bar}(\Omega) = 0$ and $\intsphere Y_1 \textit{ dA}= 0$ gives
    \begin{align}
        \intsphere ((1+u)^{n+2} -1)Y_1 \textit{dA}= 0.
    \end{align}
So,
\begin{align}
    \intsphere u Y_1 \textit{dA} = \sum_{k = 2}^{n+2}{n+2 \choose k} \intsphere u^kY_1 \textit{ dA} = O(||u||_{L^2}^2).
\end{align}
Therefore, 
\begin{align}
    a_1^2 = O(\epsilon^2)||u||_{L^2}^2.
\end{align}

Next, we consider the corresponding eigenvalue of the spherical harmonic $Y_k$, which is explicitly given by $\lambda_k = -k(k+n - 1)$. Noting that  $|\lambda_k| \geq 2(n +1) $  when $k \geq 2$, we compute 
\begin{align}
    ||\nabla u||^2_{L^2} &= \sum_{k = 1}^{\infty} 
    |\lambda_k|a_k^2
    \nonumber
    \\
    &=\sum_{k = 2}^{\infty} 
    |\lambda_k|a_k^2 + na_1^2
    \nonumber
    \\
  &  \geq 
    2(n+1)\sum_{k = 2}^{\infty} 
    a_k^2 + na_1^2
       \nonumber
    \\
  & =
    2(n+1)\sum_{k = 0}^{\infty} 
    a_k^2  -2(n+1)a_0^2  -(n +2)a_1^2
    \nonumber
    \\
    &= 2(n+1)||u||_{L^2}^2 + O(\epsilon^2)||u||_{L^2}^2+
    O(\epsilon^2)||\nabla u||_{L^2}^2.
\end{align}
\end{proof}

Next, we aim to use Proposition \ref{propgen} to bound the $(k,j)$-isoperimetric deficit below by the Frankael asymmetry $\alpha(\Omega)$. When $j=-1$ (when the volume is preserved), estimating $\alpha(\Omega)$ was reduced to being bounded above by $|\Omega \Delta B|^2$, which was bounded by $||u||_{L^2}^2$ (up to a constant). When the $\text{Vol}(\Omega) \neq \text{Vol}(B)$, estimating this quanitity is a bit more difficult, and we show in the next theorem that we can bound it above by $||\nabla u||_{L^2}^2$.

\begin{lemma}
\label{asymmetry_upperbound_novolumepreserve}
Suppose $M:= \partial \Omega$ is a nearly spherical set, then
\begin{align}
    \frac{|\Omega \Delta B_{\Omega}|^2}{|B_{\Omega}|^2}
    &\leq
    \frac{(n+1)^2}{n^2\text{Area}(\partial B)}||\nabla u||_{L^2}^2 
    +O(\epsilon)||\nabla u||_{L^2}^2,
\end{align}
where $||u||_{W^{1,\infty}}< \epsilon$.
\end{lemma}
\begin{remark}
Note that we do not need to assume $||D^2u||_{L^{\infty}} < \epsilon$ in this lemma.
\end{remark}

\begin{proof}
Recall the formula
\begin{align}
|\Omega| = \frac{1}{n+1} \intsphere (1+u)^{n+1} \textit{dA} = |B| + \sum_{k=1}
 \frac{{n+1 \choose k}}{n+1}\intsphere u^k \textit{dA} .
\end{align}

And, if $r$ is the radius of $B_{\Omega}$, then
$|\Omega| = |B_{\Omega}| = r^{n+1}|B|$. Hence,
$  r^{n+1} = \frac{|\Omega|}{|B|}$.
We compute, 
\begin{align}
    \frac{|\Omega \Delta B_{\Omega}|}{|B_{\Omega}|}
    &=
   \frac{1}{n+1} \frac{1}{|B_{\Omega}|}
    \intsphere \bigg|
    (1 + u)^{n+1} - r^{n+1}
    \bigg |\textit{dA}
    \nonumber
    \\
    &=
    \frac{1}{n+1} \frac{1}{|B_{\Omega}|}
    \intsphere \bigg|
    (1 + u)^{n+1} - \frac{|\Omega|}{|B|}
    \bigg|\textit{dA}
    \nonumber
    \\
      &=
    \frac{1}{n+1} \frac{1}{|B_{\Omega}|}
    \intsphere \bigg|
    (1 + u)^{n+1} - \frac{|B| + \sum_{k=1}
 \frac{{n+1 \choose k}}{n+1}\intsphere u^k \textit{dA} }{|B|}
 \bigg|\textit{dA}
 \nonumber\\
 &=
   \frac{1}{n+1} \frac{1}{|B_{\Omega}|}
    \intsphere \bigg|
      \sum_{k=1}^{n+1}{n +1 \choose k}u^k
    - \frac{1}{\text{Area}(\partial B)}   {n +1 \choose k}\intsphere u^k \textit{dA} 
   \bigg   |\textit{dA}
  \nonumber   \\
    &=
    \frac{1}{n+1} \frac{1}{|B_{\Omega}|}
    \intsphere  \bigg |
    \sum_{k=1}^{n+1}{n +1 \choose k}
    \bigg(
    u^k
    - 
    \text{Avg}(u^k)
    \bigg)
    \bigg |\textit{dA}
    \nonumber
   \\
       &\leq 
      \frac{1}{n+1} \frac{1}{|B_{\Omega}|}
     \sum_{k=1}^{n+1}{n +1 \choose k}
    ||u^k
    - 
    \text{Avg}(u^k)
    ||_{L^1},
\end{align}
where $\text{Avg}(u^k)$ denotes the average value of $u^k$ on $\partial B$. Then, by applying H\"{o}lder's inequality and the Poincar\'{e} inequality, we continue to bound
\begin{align}
   \frac{|\Omega \Delta B_{\Omega}|}{|B_{\Omega}|}
    &\leq 
      \frac{1}{n+1} \frac{\text{Area}(\partial B)^{1/2}}{|B_{\Omega}|}
     \sum_{k=1}^{n+1}{n +1 \choose k}
    ||u^k
    - 
    Avg(u^k)
    ||_{L^2}
 \nonumber    \\
    &\leq 
     \frac{1}{n(n+1)} \frac{\text{Area}(\partial B)^{1/2}}{|B_{\Omega}|}
     \sum_{k=1}^{n+1}{n +1 \choose k}
    ||\nabla (u^k) ||_{L^2}
   \nonumber  \\
    &\leq
     \frac{\text{Area}(\partial B)^{1/2}}{n|B_{\Omega}|}
     \sum_{k=1}^{n+1}{n \choose k-1}
     ||u^{k-1}||_{L^{\infty}}
    ||\nabla u ||_{L^2}.
\end{align}
Therefore, noting that $ ||u^{k-1}||_{L^{\infty}} = O(\epsilon)$ for $k \geq 2$,
\begin{align}
     \frac{|\Omega \Delta B_{\Omega}|^2}{|B_{\Omega}|^2}
     &\leq
    \frac{1}{n^2}  \frac{\text{Area}(\partial B)}{|B_{\Omega}|^2}
    ||\nabla u ||_{L^2}^2
    + O(\epsilon) ||\nabla u ||_{L^2}^2
   \nonumber  \\
    &= 
    \frac{(n+1)^2}{n^2\text{Area}(\partial B)}\frac{|B|^2}{| B_{\Omega}|^2} 
    ||\nabla u ||_{L^2}^2
    + O(\epsilon) ||\nabla u ||_{L^2}^2.
 \end{align}
 Then, because $\frac{|B|^2}{| B_{\Omega}|^2} = 1 + O(\epsilon)$
    \begin{align}
     \frac{|\Omega \Delta B_{\Omega}|^2}{|B_{\Omega}|^2}
    & \leq
    \frac{(n+1)^2}{n^2\text{Area}(\partial B)}
    ||\nabla u ||_{L^2}^2
    + O(\epsilon) ||\nabla u ||_{L^2}^2.
\end{align}
\end{proof}

We now prove the main theorem of this section, where we obtain a quantitative isoperimetric inequality for the $(k,j)$-isoperimetric deficit. Recall from Section \ref{deficitprelim} that normalizing $\Omega$ such that $I_j(\Omega) = I_j(B)$ yields
\begin{align}
    \delta_{k,j}(\Omega) = \frac{I_k(\Omega)  - I_k(B)}{I_k(B)}.
\end{align}

\vspace{.03in}
\noindent \textbf{Theorem \ref{secondthm}.}
\textit{Fix $0 \leq j < k$. Suppose $\Omega= \{ (1+u(\frac{x}{|x|}))x: x \in B  \}\subseteq \mathbb{R}^{n+1}$, where $u \in C^3(\partial B)$,  $I_j(\Omega) = I_j(B)$, and $\text{bar}(\Omega) =0$.}
 \textit{For all $\eta>0$, there exists $\epsilon>0$ such that if  $||u||_{W^{2,\infty}}< \epsilon$, then}
\begin{align}
    \delta_{k,j}(\Omega)
    \geq 
\bigg (
\frac{n(n-k)(k-j)}{4(n+1)^2}  -\eta \bigg)\alpha^2(\Omega).
\end{align}

\begin{proof}
Suppose $||u||_{W^{2,\infty}}< \epsilon$. 
Applying Lemma \ref{asymmetry_upperbound_novolumepreserve},

\begin{align}
     \alpha^2(\Omega)
      &\leq 
   \frac{|\Omega \Delta B|^2}{|B|^2}
     \leq 
     \bigg (\frac{(n+1)^2}{n^2\text{Area}(\partial B)} + O(\epsilon)
     \bigg )||\nabla u||^2.
\end{align}
Thus
  \begin{align}
      ||\nabla u||_{L^2}^2 \geq \bigg (\frac{n^2}{(n+1)^2}\text{Area}(\partial B)+ O(\epsilon) \bigg )\alpha^2(\Omega). 
        \end{align} 
 Additionally, since $I_j(\Omega) = I_j(B)$,
 \begin{align}
    \delta_{k,j}(\Omega) = \frac{ I_k(\Omega) - I_k(B)}{I_k(B)}= \frac{ I_k(\Omega) - I_k(B)}{{n \choose k}\text{Area}(\partial B)} .
      \end{align} 
      Therefore, applying Proposition \ref{propgen}, when $||u||_{W^{2,\infty}}< \epsilon$, we have that
\begin{align}
   \delta_{k,j}(\Omega)
    &
    \geq \bigg (
\frac{(n-k)(k-j)}{4n\text{Area}(\partial B)}  + O(\epsilon) \bigg)||\nabla u||_{L^2}^2
\nonumber
\\
&\geq
\bigg (
\frac{(n-k)(k-j)}{4n\text{Area}(\partial B)}  + O(\epsilon) \bigg)\bigg (\frac{n^2}{(n+1)^2}\text{Area}(\partial B) + O(\epsilon) \bigg )\alpha^2(\Omega) 
\nonumber
\\
&=
\bigg (
\frac{n(n-k)(k-j)}{4(n+1)^2}  + O(\epsilon) \bigg)\alpha^2(\Omega).
\end{align}

\end{proof}

\section{Bounds on $||u||_{L^{\infty}}$}
\label{haussdorffsection}
Following the argument of Fuglede in \cite{MR942426} for stability of  the classical isoperimetric inequality, we control $||u||_{L^{\infty}}$ using the $(k,\minus 1)$-isoperimetric deficit. Because we consider $\Omega$ when $\text{Vol}(\Omega) = \text{Vol}(B)$ and $\text{bar}(\Omega) = 0$, $||u||_{L^{\infty}}$ is simply the spherial deviation $d(\Omega)$ from Definition \ref{sphericaldeviation}.   First, we state a lemma from \cite{MR942426}. 
\begin{lemma}  
(Fuglede \cite{MR942426}, Lemma 1.4) Suppose $w: \partial B \rightarrow \mathbb{R}$ is a Lipschitz function where $\intsphere w \textit{dA} = 0$. Then
  \[ ||w||_{L^{\infty}}^n 
  \leq
  \begin{cases} 
          \pi || \nabla w||_{L^1} \leq \pi || \nabla w||_{L^2} & n = 1 \\
        4|| \nabla w||_{L^2}^2 \log \frac{8e|| \nabla w||_{L^{\infty}}^2}{|| \nabla w||_{L^2}^2} & n=2\\
          C|| \nabla w||_{L^2}^2|| \nabla w||_{L^{\infty}}^{n-2} & n \geq 3,
       \end{cases}
    \]
\label{zeromean}
where $C>0$ depends only on $n$.
\end{lemma}
This lemma is useful because the assumption that $\text{Vol}(\Omega) = \text{Vol}(B)$ is equivalently stated expressed as:
\begin{align}
    \intsphere (1+u)^{n+1} - 1 \textit{dA} =0.
\end{align}
So, we set $w = \frac{1}{n+1}((1+u)^{n+1} - 1 )$ and apply Lemma \ref{zeromean} to $w$. 
\vspace{.1in}

\noindent\textbf{Theorem \ref{infboundintro}.}
\textit{ Suppose  $\Omega = \{ (1+u(\frac{x}{|x|}))x: x \in B  \} \subseteq \mathbb{R}^{n+1}$, where $u \in C^3(\partial B)$, $\text{Vol}(\Omega) = \text{Vol}(B)$, and $\text{bar}(\Omega)=0$.}
 \textit{There exists an $\eta >0$ so if  $||u ||_{W^{2,\infty}}< \eta$, then}
 \[ ||u||_{L^{\infty}}^n 
  \leq
  \begin{cases} 
          C \delta_{k,- 1}^{1/2}(\Omega) & n = 1 \\
        C \delta_{k,-1}(\Omega)\log \frac{A}{\delta_{k,-1}(\Omega)} & n=2\\
          C\delta_{k,-1}(\Omega) & n \geq 3 
       \end{cases}
    \]
  \textit{ where $A,C>0$ depend only on $n,k$.}

\begin{proof}
As noted above, setting $w = \frac{1}{n+1}( (1+u)^{n+1} - 1)$ gives
\begin{align}
    \intsphere w \textit{dA} = 0.
\end{align}
Therefore, Lemma \ref{zeromean} applies to $w$. Moreover, as shown in \cite{MR942426}, there exists an $\eta>0$ such that when $||u||_{W^{2,\infty}}< \beta$, then
\begin{align}
    (1 - O(\eta))|u| \leq |w| \leq    (1 + O(\eta))|u|,
\end{align}
and
\begin{align}
    (1 - O(\eta))|\nabla u| \leq |\nabla w| \leq    (1 +O(\eta))|\nabla u|.
\end{align}
First we suppose $n\geq 3$. We will then prove the theorem for $n=2$, and $n=1$ follows similarly. Applying Lemma \ref{zeromean}, there is a constant $C>0$ (possibly changing from line to line) where
\begin{align}
 \label{boundfromzeromean} 
    || \nabla u ||_{L^2}^2
    &\geq  C|| \nabla w ||_{L^2}^2
  \geq  \frac{C||w||_{\infty}^{n}}{C(n) || \nabla w ||^{n-2}_{\infty}} 
   \geq  \frac{C||u||_{L^{\infty}}^{n}}{C(n) || \nabla u ||^{n-2}_{\infty}} 
    \geq C ||u||_{L^{\infty}}^{n}.
  \end{align}
 By Proposition \ref{propvol},  for small enough $||D^2u||_{L^{\infty}}$, 
\begin{align}
    \delta_{k,\minus 1}(\Omega) \geq C ||\nabla u||_{L^2}^2,
\end{align}
which together with \eqref{boundfromzeromean} gives the statement of the theorem for $n\geq 3$.

\noindent Next suppose $n=2.$ Then, there is a $M>0$ such that
\begin{align}
    || \nabla u ||_{L^2}^2 \log  \frac{M||\nabla u||_{L^{\infty}}^2}{||\nabla u||_{L^2}^2} 
    & \geq
    C|| \nabla w ||_{L^2}^2 \log  \frac{8e||\nabla w||_{L^{\infty}}^2}{||\nabla  w||_{L^2}^2} 
  \nonumber
  \\
    &
    \geq
    C||w||_{L^{\infty}}^n
    \nonumber
    \\
    &
    \geq C||u||_{L^{\infty}}^n.
    \label{firstn=2}
\end{align}
Furthermore,
\begin{align}
     || \nabla u ||_{L^2}^2 \log  \frac{M||\nabla u||_{L^{\infty}}^2}{||\nabla u||_{L^2}^2} 
     &\leq
     C_1\delta_{k,-1}(\Omega)\log \frac{C_2}{||\nabla u||_{L^2}^2} 
     \nonumber
     \\
     &\leq
      C_1\delta_{k,-1}(\Omega)\log  \frac{C_2}{\delta_{k, -1}(\Omega)}. 
      \label{secondn=2}
\end{align}
The last line follows from the observation in \eqref{Ikbeforespherical}, where for sufficiently small $\|u\|_{W^{2, \infty}}$ we have $\delta_{k,\minus1} \leq C(n,k) ||\nabla u||_{L^2}^2$ for some positive constant $C(n,k)>0$. Combining \eqref{firstn=2} and \eqref{secondn=2} concludes the statement of the theorem for $n=2$. The proof for $n=1$ follows similarly.
\end{proof}

\bibliography{main}

\end{document}